\documentclass[11pt]{amsart}

\usepackage{amsmath, amsfonts, amsthm, amssymb, array, xcolor, graphicx, mathtools, url, hyperref}

\usepackage{fullpage}

\newtheorem{theorem}{Theorem}[section]

\newtheorem{lemma}[theorem]{Lemma}

\numberwithin{theorem}{section}

\newcommand{\C}{\mathbb{C}}
\newcommand{\N}{\mathbb{N}}

\begin{document}

\title{Toward Resolving Kang and Park's generalization of the Alder-Andrews theorem}

\author{Leah Sturman}
\address{Southern Connecticut State University}
\email{sturmanL1@southernct.edu}
\author{Holly Swisher}
\address{Oregon State University}
\email{swisherh@oregonstate.edu}

\thanks{This work was supported by NSF grant DMS-2101906.}



\begin{abstract}
The Alder-Andrews Theorem, a partition inequality generalizing Euler's partition identity, the first Rogers-Ramanujan identity, and a theorem of Schur to $d$-distinct partitions of $n$, was proved successively by Andrews in 1971, Yee in 2008, and Alfes, Jameson, and Lemke Oliver in 2010.  While Andrews and Yee utilized $q$-series and combinatorial methods,  Alfes et al. proved the finite number of remaining cases using asymptotics originating with Meinardus together with high-performance computing.  In 2020, Kang and Park conjectured a ``level $2$" Alder-Andrews type partition inequality which relates to the second Rogers-Ramanujan identity.  Duncan, Khunger, the second author, and Tamura proved Kang and Park's conjecture for all but finitely many cases using a combinatorial shift identity.  Here, we generalize the methods of Alfes et al. to resolve nearly all of the remaining cases of Kang and Park's conjecture.   
\end{abstract}

\maketitle

\section{Introduction and Statement of Results}

A partition of a positive integer $n$ is a non-increasing sequence of positive integers, called parts, that sum to $n$.  The study of partition counting functions has famously revealed deep connections with many important areas of mathematics and mathematical physics through its connections to automorphic forms and representation theory (see \cite{andrews_book} and \cite{BFOR} for some examples). We write $p(n | \text{condition})$ to denote the number of partitions of $n$ satisfying a specified condition, and define
\begin{align}
q_d^{(a)}(n) &:= p(n\, | \text{ parts} \geq a \text{ and differ by at least } d), \label{RR1} \\
Q_d^{(a)}(n)  &:= p(n\, | \text{ parts} \equiv \pm a \!\!\! \pmod{d+3}), \label{RR2} \\
\Delta_d^{(a)}(n) &:= q_d^{(a)}(n)-Q_d^{(a)}(n).
\end{align}

This notation allows us to write Euler's partition identity, which states that the number of partitions of $n$ into distinct parts equals those into odd parts, as $\Delta_1^{(1)}(n)=0$.  Similarly, the celebrated first and second Rogers-Ramanujan identities, written here in terms of $q$-Pochhammer notation\footnote{Here $(a;q)_0:=1$ and $(a;q)_n:=\prod_{k=0}^{n-1}(1-aq^k)$ for $1\leq n\leq \infty$.} as
\begin{align*}
\sum_{n=0}^\infty\frac{q^{n^2}}{(q;q)_n} &=\frac{1}{(q;q^5)_\infty(q^4;q^5)_\infty},\\
\sum_{n=0}^\infty\frac{q^{n^2+n}}{(q;q)_n} &=\frac{1}{(q^2;q^5)_\infty(q^3;q^5)_\infty},
\end{align*}
can be written as $\Delta_2^{(1)}(n)=0$ and $\Delta_2^{(2)}(n)=0$, respectively.

Motivated by these identities, Schur \cite{SCHUR} proved that the number of partitions of $n$ into parts differing by at least $3$, where no two consecutive multiples of 3 appear, equals the number of partitions of $n$ into parts congruent to $\pm1\pmod{6}$, which implies that $\Delta_3^{(1)}(n)\geq 0$. 

After Lehmer \cite{Lehmer} and Alder \cite{Alder_nonex} proved that no other partition identities for $q_d^{(a)}(n)$ analogous to the Rogers-Ramanujan identities can exist, in 1956 Alder \cite{Alder_conj} conjectured a different type of generalization.  Namely, that for all positive integers $n,d\geq 1$,
\begin{equation}\label{Alder_conj}
\Delta_d^{(1)}(n)\geq 0.
\end{equation}

In 1971, Andrews \cite{Andrews_Alder} proved \eqref{Alder_conj} when $d=2^k-1$ and $k\geq 4$.  Later, in 2004 Yee \cite{Yee_04,Yee_08} proved \eqref{Alder_conj} for all $d\geq 32$ and $d=7$.  Both Andrews and Yee used $q$-series and combinatorial methods.  In 2011, Alfes, Jameson, and Lemke Oliver \cite{AJLO} completely resolved the conjecture using asymptotic methods originating with Meinardus \cite{Meinardus_asymptotics, Meinardus_partitions}, together with detailed computer programming and high-performance computing to prove the remaining cases.

In 2020, Kang and Park \cite{KangPark} posed the question of whether \eqref{Alder_conj} can be generalized to incorporate the second Rogers-Ramanujan identity.  They observed that $\Delta_d^{(2)}(n)$ is negative for some $n,d\geq 1$, but observed that removing one part in the calculation of $Q_d^{(2)}(n)$ by defining
\begin{align*}
Q_d^{(a,-)}(n)  &:= p(n\, | \text{ parts} \equiv \pm a \!\!\! \pmod{d+3},\text{ excluding the part } d+3-a), \\
\Delta_d^{(a,-)}(n) &:= q_d^{(a)}(n)-Q_d^{(a,-)}(n),
\end{align*}
appeared to suffice, where by definition $\Delta_d^{(2,-)}(n)\geq \Delta_d^{(2)}(n)$.  Kang and Park's conjecture states that for all $n,d\geq 1$,
\begin{equation}\label{KP_conj}
\Delta_d^{(2,-)}(n)\geq 0.
\end{equation}

Kang and Park \cite{KangPark} proved \eqref{KP_conj} when $n$ is even and $d=2^k-2$ for $k\geq 5$ or $k=2$.  Then in 2021, Duncan, Khunger, the second author, and Tamura \cite{DKST} proved \eqref{KP_conj} for all $d\geq 62$.  Since $\Delta_2^{(2)}(n)=0$, this leaves the remaining cases of $d=1$ and $3\leq d\leq 61$.  Here, we prove all but three of these cases.

\begin{theorem}\label{fullKP}
Kang and Park's conjecture \eqref{KP_conj} is true for $6\leq d \leq 61$ and $d=1$.
\end{theorem}

A first approach in using asymptotics to prove Theorem \ref{fullKP} would be to obtain explicit asymptotics for the functions $q_d^{(2)}(n)$ and $Q_d^{(2,-)}(n)$, determine a bound $N(d)$ such that for any $n\geq N(d)$ it follows that $q_d^{(2)}(n) \geq Q_d^{(2,-)}(n)$, and use computing to show $q_d^{(2)}(n) \geq Q_d^{(2,-)}(n)$ holds for all $n<N(d)$.  However, while generalizing Alfes et al. \cite[Thm. 3.1]{AJLO} to $q_d^{(2)}(n)$ is straightforward (see Theorem \ref{q2asymptotics}), approaching $Q_d^{(2,-)}(n)$ is more difficult.  One way of avoiding this altogether is to observe that by a theorem of Andrews \cite[Thm. 3]{Andrews_Alder} it follows that $Q_d^{(1)}(n) \geq Q_d^{(2,-)}(n)$ for all $d,n\geq 1$.  Thus when $d\geq 4$ we can use the existing asymptotics of $Q_d^{(1)}(n)$ given by Alfes et al. \cite[Thm. 2.1]{AJLO} together with Theorem \ref{q2asymptotics} to obtain a bound past which we can guarantee that 
\begin{equation} \label{method1} 
q_d^{(2)}(n)\geq Q_d^{(1)}(n)\geq Q_d^{(2,-)}(n).
\end{equation} 
We do in fact use this method when $d$ is odd.  However, when $d$ is even the bounds we obtained were not sufficient for computations using the High Performance Computing Cluster (HPC) at Oregon State University.  Instead when $d$ is even, we use the trivial fact that $Q_d^{(2)}(n) \geq Q_d^{(2,-)}(n)$ for all $d,n\geq 1$ to approach the problem by way of asymptotics for $Q_d^{(2)}(n)$.  In particular, we prove the following result by modifying the approach of Alfes et al. in \cite[Thm. 2.1]{AJLO}.  

\begin{theorem}\label{Q2asymptotics}
For $d\geq 4$ even and $n\geq 1$,
\[
Q_{d}^{(2)}(n)=\frac{1}{4(3(d+3))^{\frac{1}{4}}\sin(\frac{2\pi}{d+3})} n^{-\frac{3}{4}} \exp{\left(\frac{2\pi \sqrt{n}}{\sqrt{3(d+3)}}\right)}+R_d(n),
\]
where $R_d(n)$ is an explicitly bounded function described in \S \ref{bounds}.
\end{theorem}

Armed with Theorem \ref{Q2asymptotics}, for even $d\geq 4$ we can obtain a bound past which we can guarantee
\begin{equation} \label{method2} 
q_d^{(2)}(n)\geq Q_d^{(2)}(n)\geq Q_d^{(2,-)}(n),
\end{equation} 
and the bounds produced for the range of $d$ we are considering are sufficient for our computations.  

When $d=4,5$ the bounds obtained using the methods above are not sufficient for our computations, and when $d<4$ a different approach is needed since Theorems \ref{q2asymptotics} and \ref{Q2asymptotics} don't apply.

As a consequence of our computations, we obtain the following result.

\begin{theorem}\label{CKK_extension}
    For $6\leq d\leq 61$,
    \begin{equation*}
        \Delta_d^{(2)}(n)\geq 0,
    \end{equation*} for all $n\geq 0$ when $d$ is even and for all $n\geq 0$ except $n=d+1,d+3,d+5$ when $d$ is odd. 
\end{theorem}
We note that Theorem \ref{CKK_extension} proves additional cases of a recent Theorem of Cho, Kang, and Kim \cite[Theorem 1.1]{CKK}.  Their theorem  holds for $d=126$, and $d\geq 253$. 

The rest of this paper is organized as follows.  In Section \ref{preliminaries}, we prove the $d=1$ case of Theorem \ref{fullKP} using combinatorial methods and state some needed results from the work of Alfes et al. \cite{AJLO}.  We then prove Theorem \ref{Q2asymptotics} in Section \ref{asymptotics}.  In Section \ref{bounds}, we describe how we obtain explicit bounds $N(d)$ which guarantee $\Delta_d^{(2,-)}(n)\geq 0$ for all $n\geq N(d)$ when $4\leq d \leq 61$.  In Section \ref{computations}, we discuss our computations to show that $\Delta_d^{(2,-)}(n)\geq 0$ for all $n \leq N(d)$ when $4\leq d \leq 61$, and complete the proofs of Theorem \ref{fullKP} and Theorem \ref{CKK_extension}.  Lastly in Section \ref{conclusion}, we conclude with a brief discussion on possible future work.

\section{Preliminaries}\label{preliminaries} 

We first give an asymptotic formula for $q_d^{(2)}(n)$ when $d\geq 4$.

\begin{theorem}\label{q2asymptotics}
Let $d\geq 4$, and $\alpha$ the unique real solution of $x^d+x-1=0$ in the interval $(0,1)$. Let $A_d:=\frac{d}{2}\log ^2\alpha+\sum_{r=1}^{\infty}\frac{\alpha^{rd}}{r^2}$.  For positive integers $n\geq 1$,
\begin{equation*}
    q_d^{(2)}(n)=\frac{A_d^{1/4}}{2\sqrt{\pi\alpha^{d-3}(d\alpha^{d-1}+1)}}n^{-\frac{3}{4}}\exp\left(2\sqrt{A_d n}\right)+r_d(n),
\end{equation*}
where $|r_d(n)|$ is an explicitly bounded function described in \S \ref{bounds}.
\end{theorem}

As this arises from a straightforward generalization of Alfes et al. \cite[Thm. 3.1]{AJLO} using work of Meinardus \cite{Meinardus_partitions} and is previously described by Duncan et al. \cite[Thm. 6.5]{DKT_proc}, we omit a proof. 

\subsection{Proof of the $d=1$ case of Theorem \ref{fullKP}} \label{ss:d=1}

To show that $\Delta^{(2,-)}_1(n) \geq 0$ for all positive integers $n$ we use a result of Duncan et al. \cite[Lemma 2.4]{DKST} which states that for $a,d \geq 1$, and $n \geq d+2a$,

\begin{equation}\label{DKSTlem2.4}
q_d^{(a)}(n) \geq q_{\left\lceil \frac{d}{a}\right\rceil}^{(1)} \left( \left\lceil \frac{n}{a} \right\rceil \right).
\end{equation}

\begin{proof}[Proof of Theorem \ref{fullKP}, d=1] First, when $n$ is odd, $Q_1^{(2,-)}(n)$ is trivially zero and we are done. Now suppose $n$ is even.  By \eqref{DKSTlem2.4}, when $n\geq 5$, $q_1^{(2)}(n)\geq q_1^{(1)}\left(\frac{n}{2}\right)$. And by Euler's partition identity, $q_1^{(1)}\left(\frac{n}{2}\right)=Q_1^{(1)}\left(\frac{n}{2}\right)$.  Moreover $Q_1^{(1)}\left(\frac{n}{2}\right)=Q_1^{(2)}(n)$, since the partitions of $\frac{n}{2}$ into parts congruent to $\pm 1$ modulo 4 are in bijection with partitions of $n$ into parts congruent to $\pm 2$ modulo $4$.  Putting this together, we have for even $n\geq 6$,
\begin{equation*}
    q_1^{(2)}(n)\geq q_1^{(1)}\left(\frac{n}{2}\right)=Q_1^{(1)}\left(\frac{n}{2}\right)=Q_1^{(2)}(n)\geq Q_1^{(2,-)}(n).
\end{equation*}
A quick check that $q_1^{(2)}(n)\geq Q_1^{(2,-)}(n)$ when $n=2,4$ completes the proof.
\end{proof}

We note that when $d=3$ this method fails for even $n$.  Using \eqref{DKSTlem2.4} and \eqref{RR1} we obtain for $n\geq 7$,
\[
q_3^{(2)}(n)\geq q_2^{(1)}\left(\frac{n}{2}\right)=Q_2^{(1)}\left(\frac{n}{2}\right).
\]
But here $Q_2^{(1)}(\frac{n}{2})=Q_{7}^{(2)}(n)$, since partitions of $\frac{n}{2}$ into parts congruent to $\pm 1$ modulo $5$ are in bijection with partitions of $n$ into parts congruent to $\pm 2$ modulo $10$.  However, $Q_{7}^{(2)}(n)\leq Q_{3}^{(2,-)}(n)$ by Duncan et al. \cite[Lemma 2.2]{DKST}.

Indeed, this method fails to generalize to higher values of odd $d$, and doesn't work for even $d$ when $n$ is odd, so we return to asymptotic methods to approach \eqref{KP_conj}.  

\subsection{Some useful estimates}

To obtain the bounds $N(d)$ in Section \ref{bounds} we follow the approach of Alfes et al. \cite{AJLO}, which require specific estimation of error terms of $q_d^{(2)}(n)$.  The estimates in the following lemma are used in Section \ref{bounds}. 

Recall the Hurwitz zeta function, $\zeta(s,a)$, is defined when $\sigma >1$ and $a\neq 0,-1,-2,\dots$ by 
\begin{equation}\label{Hzeta}
    \zeta(s,a)=\sum_{n=0}^{\infty}\frac{1}{(n+a)^s},
\end{equation} 
and is analytically continued to a meromorphic function on $\C$ having a single pole of order 1 at $s=1$ with residue $1$.

\begin{lemma}\label{q2preliminaries}
    Let $\alpha$ be the unique real solution of $x^d+x-1=0$ in the interval $(0,1)$, and let $A_d:=\frac{d}{2}\log ^2\alpha+\sum_{r=1}^{\infty}\frac{\alpha^{rd}}{r^2}$.  Set $\rho=\alpha^d=1-\alpha$, $0<\xi<1$, $0<\varepsilon<\frac{1}{2}$, and $x<\beta$ where
    \begin{equation*}
        \beta:=\min \left(\frac{-\pi}{\log\rho}\xi,\ \frac{2\alpha^{2-d}}{\pi d},\ \frac{1}{2d}+\rho\left(\frac{1}{2}-\frac{\pi^2}{24}\right)\right)^{\frac{1}{\varepsilon}}.
    \end{equation*}
    Then, 
    \begin{equation*}
        f_1(x):=(1+x^{2\varepsilon})^{\frac{1}{4}}2^{-\frac{1}{2}}\pi^{-\frac{3}{2}}\zeta\left(\frac{3}{2},2\right) \left(\frac{\rho}{1-\rho}\right) \left( \frac{1}{\frac{\pi}{2}-\arctan x^{\varepsilon}}\right)
    \end{equation*}
    and 
    \begin{multline*}
    f_2(x):=e^{\frac{d|x|}{8}}\left[ e^{\frac{dx\sqrt{1+x^{2\varepsilon}}}{8}} -1+\frac{2\exp\left(-\frac{4\pi^2(1-\xi)}{dx(1+x^{2\varepsilon})}\right)}{1-\exp\left(-\frac{2\pi^2(1-\xi)}{dx(1+x^{2\varepsilon})}\right)}\right] \\
     +2\exp\left( \frac{2\pi(|\rho|-\pi)}{dx(1+x^{2\varepsilon})}-\frac{2\log\rho}{d}x^{\varepsilon-1}+\frac{d|x|}{8}\right)
    \end{multline*}
are bounded when $x=\sqrt{\frac{A_d}{n}}$ as functions of $n\in \N$ by the explicit constant upper bounds given in \eqref{F1} and \eqref{F2} respectively. 
\end{lemma}

\begin{proof}
    First, we bound $f_1(x)$. Substituting $x=\sqrt{\frac{A_d}{n}}$ gives
    \begin{equation*}
    f_1\left(\sqrt{\frac{A_d}{n}}\right)= (1+A_d^{\varepsilon}n^{-\varepsilon})^{\frac{1}{4}}   2^{-\frac{1}{2}}    \pi^{-\frac{3}{2}}\zeta\left(\frac{3}{2},2\right) \left(\frac{\rho}{1-\rho}\right) \left(\frac{1}{\frac{\pi}{2}-\arctan(A_d^{\frac{\varepsilon}{2}}n^{-\frac{\varepsilon}{2}})}\right).   
    \end{equation*}
    Noting that $(1+A_d^{\varepsilon}n^{-\varepsilon})^{\frac{1}{4}}$ and $(\frac{\pi}{2}-\arctan(n^{-\frac{\color{blue}\varepsilon\color{black}}{2}}A_d^{\frac{\varepsilon}{2}}))^{-1}$ are decreasing in $n$, and $n\geq 1$, we obtain the constant bound $F_1$ given by
    \begin{equation}\label{F1}
        f_1\left(\sqrt{\frac{A_d}{n}}\right)\leq F_1:=(1+A_d^{\varepsilon})^{\frac{1}{4}}   2^{-\frac{1}{2}}    \pi^{-\frac{3}{2}}\zeta\left(\frac{3}{2},2\right) \left(\frac{\rho}{1-\rho}\right) \left(\frac{1}{\frac{\pi}{2}-\arctan(A_d^{\frac{\varepsilon}{2}})}\right).
    \end{equation}

    Next, we show that $f_2(x)$ is increasing in $x$ so that $f_2(\sqrt{\frac{A_d}{n}})$ is decreasing in $n$.  First, observe that as $x$ increases, $-\frac{4\pi^2(1-\xi)}{dx(1+x^{2\varepsilon})}$ increases towards $0$, and thus $\exp\left(-\frac{4\pi^2(1-\xi)}{dx(1+x^{2\varepsilon})}\right)$ increases towards $1$.  Similarly, in the denominator, $1-\exp\left(\frac{2\pi^2(1-\xi)}{dx(1+x^{2\varepsilon})}\right)$ decreases toward $0$ as $x$ increases.  As $e^{\frac{d|x|}{8}}$ and $e^{\frac{2x\sqrt{1+x^{2\varepsilon}}}{8}}$ are increasing functions of $x$ we have that 
    \begin{equation*}
        e^{\frac{d|x|}{8}}\left[ e^{\frac{dx\sqrt{1+x^{2\varepsilon}}}{8}} -1+2\frac{\exp\left(-\frac{4\pi^2(1-\xi)}{dx(1+x^{2\varepsilon})}\right)}{1-\exp\left(-\frac{2\pi^2(1-\xi)}{dx(1+x^{2\varepsilon})}\right)}\right] 
    \end{equation*} is an increasing function of $x$ as well. 
    
    Observing that $\rho \in (0,1)$ we see that both $\frac{2\pi(|\rho |-\pi)}{dx(1+x^{2\varepsilon})}$ and $\frac{d|x|}{8}$ are increasing in $x$, and $-\frac{2\log \rho}{d}x^{\varepsilon-1}$ is decreasing toward $0$. Thus, 
    \begin{equation*}
        2\exp\left( \frac{2\pi(|\rho|-\pi)}{dx(1+x^{2\varepsilon})}-\frac{2\log\rho}{d}x^{\varepsilon-1}+\frac{d|x|}{8}\right)
    \end{equation*}
    is an increasing function of $x$.  Together, we have
    \begin{equation}\label{F2}
        f_2\left(\sqrt{\frac{A_d}{n}}\right)\leq F_2:=f_2(\sqrt{A_d}).
    \end{equation}
\end{proof}

\section{Proof of Theorem \ref{Q2asymptotics}}\label{asymptotics}

For this section we write $\tau=y+2\pi i x$ where $y=Re(\tau)>0$, and set $q=e^{-\tau}=e^{-y-2\pi ix}$. Our proof follows the method of Alfes et al. \cite[Thm. 2.1]{AJLO}.  Write
\begin{equation}\label{def f(tau)}
f(\tau)=\!\!\! \prod_{\substack{n\geq 1 \\ n\equiv \pm 2(d+3)}} \!\!\! \frac{1}{(1-q^n)}=1+\sum_{n=1}^{\infty}Q_{d+3}^{(2)}(n)q^n.
\end{equation} 
Then we associate to $f$ the following Dirichlet series in $s=\sigma + it$,
\begin{equation}\label{Ddef}
D(s)= \!\!\! \sum_{\substack{n\geq 1 \\ n\equiv \pm 2(d+3)}} \!\!\! \frac{1}{n^s},
\end{equation} 
which converges for $\sigma >1$.  We can write $D(s)$ directly in terms of the Hurwitz zeta function $\zeta(s,a)$, which is defined when $\sigma >1$ and $a\neq 0,-1,-2,\dots$ by 
\begin{equation*}
    \zeta(s,a)=\sum_{n=0}^{\infty}\frac{1}{(n+a)^s},
\end{equation*} 
and is analytically continued to a meromorphic function on $\C$ with a single pole of order 1 at $s=1$ with residue $1$.  Moreover, $\zeta(s,a)$ satisfies the identities $\zeta(0,a)=\frac{1}{2}-a$ (see Apostol \cite{Apostol}) and $\zeta'(0,a) = \log(\Gamma(a)) - \frac12\log(2\pi)$ (see Lerch \cite{whittaker1927course}).

We see directly that
\begin{equation*}
    D(s)=(d+3)^{-s}\left(\zeta\left(s,\frac{2}{d+3}\right) +\zeta\left(s,\frac{d+1}{d+3}\right)\right).
\end{equation*}
Thus, $D(s)$ can be analytically continued to a meromorphic function on $\C$ with a single pole of order 1 at $s=1$ with residue $\frac{2}{d+3}$.  Moreover, $D(0)=0$, and using the reflection formula for the $\Gamma$-function, we observe that 
\begin{equation} \label{eq:D'(0)}
 D'(0)  = \zeta'\left(0,\frac{2}{d+3}\right)+\zeta'\left(0,\frac{d+1}{d+3}\right) = \log\bigg(\frac{1}{2\sin\left( \frac{2\pi}{d+3}\right)} \bigg).
\end{equation}

It will be useful to define the following function $g(\tau)$, which depends on $d$, by
\begin{equation*}
    g(\tau)= \!\!\! \sum_{\substack{n\equiv \pm 2(d+3)\\ n\geq 0}} \!\!\! q^n.
\end{equation*}
The following lemma will be needed in our proof of Theorem \ref{Q2asymptotics}.

\begin{lemma}\label{lemma Ti}
Let $\tau=y+2\pi i x$ with $y=Re(\tau)>0$, and $q=e^{-\tau}=e^{-y-2\pi ix}$.  If $\arg(\tau)>\frac{\pi}{4}$ and $|x|\leq \frac{1}{2},$ then for $d\geq 4$ even,
\begin{equation*}
    Re(g(\tau))-g(y)\leq -c_2y^{-1},
\end{equation*} where $c_2$ is an explicit constant depending only on $d$.
\end{lemma}
\begin{proof}
Note that
\begin{multline*}
g(\tau)=\sum_{n\geq 0}q^{(d+3)n+2}+\sum_{n\geq 0}q^{(d+3)n+d+1}
=(q^2+q^{d+1})\sum_{n\geq 0}(q^{(d+3)})^n\\
=\frac{q^2+q^{(d+1)}}{1-q^{d+3}}
=\frac{e^{-2\tau}+e^{-(d+1)\tau}}{1-e^{-(d+3)\tau}}
=\frac{e^{(d+1)\tau}+e^{2\tau}}{e^{(d+3)\tau}-1}.
\end{multline*}
Plugging in $\tau=y+2\pi i x$, we expand $g(\tau)$ as 
\begin{align*}
    g(\tau) = \frac{e^{(d+1)\tau}+e^{2\tau}}{e^{(d+3)\tau}-1}&=\frac{e^{(d+1)(y+2\pi i x)}+e^{2(y+2\pi i x)}}{e^{(d+3)(y+2\pi i x)}-1}\cdot \frac{e^{(d+3)(y-2\pi i x)}-1}{e^{(d+3)(y-2\pi i x)}-1}\\
    &=\frac{e^{(2d+4)y-2(2\pi i x)}+e^{(d+5)y-(d+1)(2\pi i x)}-(e^{(d+1)(y+2\pi i x)}+e^{2(y+2\pi i x)})}{e^{2(d+3)y}-2e^{(d+3)y}\cos (2\pi x(d+3))+1}.
\end{align*}
Thus, \begin{equation*}
   Re(g(\tau))=\frac{\left(e^{(2d+4)y}-e^{2y}\right)\cos(4\pi  x)+\left(e^{(d+5)y}-e^{(d+1)y}\right)\cos(2\pi (d+1)x)}{e^{2(d+3)y}-2e^{(d+3)y}\cos (2\pi x(d+3))+1}. 
\end{equation*}
Now, consider $-y(Re(g(\tau))-g(y)).$ Expanding, we find that
\begin{equation*}
    -y(Re(g(\tau))-g(y))=T_1+T_2+T_3
\end{equation*} where
\begin{equation*}
    T_1:=\frac{(1-\cos(4\pi x))(e^{(3d+7)y}+e^{2y}-e^{(2d+4)y}-e^{(d+5)y})}{\left(\frac{e^{(d+3)y}-1}{y}\right)\left(e^{2(d+3)y}-2e^{(d+3)y}\cos (2\pi x(d+3))+1\right)},
\end{equation*}

\begin{equation*}
    T_2:=\frac{(1-\cos(2\pi(d+1)x))(e^{(2d+8)y}+e^{(d+1)y}-e^{(2d+4)y}-e^{(d+5)y})}{\left(\frac{e^{(d+3)y}-1}{y}\right)\left(e^{2(d+3)y}-2e^{(d+3)y}\cos (2\pi x(d+3))+1\right)},
\end{equation*}

and 

\begin{equation*}
    T_3:=\frac{ 2(1-\cos(2\pi(d+3)x))(e^{(2d+4)y}+e^{(d+5)y})}{\left(\frac{e^{(d+3)y}-1}{y}\right)\left(e^{2(d+3)y}-2e^{(d+3)y}\cos (2\pi x(d+3))+1\right)}.
\end{equation*}
It follows that $T_1, T_2, T_3$ are all nonnegative. Furthermore, letting $y\rightarrow 0,$ we obtain that $T_1\rightarrow 0,$ $T_2\rightarrow 0$, and $T_3\rightarrow \frac{2}{d+3}$.  Thus to bound $T_1+T_2+T_3$ away from $0$ it suffices to bound any one of $T_1, T_2, T_3$ away from $0$. We observe that it is necessary for $d$ to be even, since if $d$ is odd, then when $|x|=\frac{1}{2}$ each of $T_1,T_2,T_3$ are simultaneously zero.  Since each of $T_1,T_2,T_3$ is an even function in $x$, we may assume $x$ is nonnegative.  Further, since $\arg (\tau)>\frac{\pi}{4}$, we know that $y<2\pi x$.

{\bf Case 1.}  Suppose that $\frac{d+4}{2(d+5)} \leq x \leq \frac12$.  Since $d$ is even and $d+2 \leq \frac{(d+4)(d+3)}{d+5}$, we have for $x$ in this range that $\cos(2(d+3)\pi x)$ is decreasing.  Thus $1-\cos(2(d+3)\pi x)$ is increasing and has its minimum at $x=\frac{d+4}{2(d+5)}$.  Since $d$ is even, $\cos(\frac{(d+3)(d+4)}{(d+5)}\pi) = \cos(\frac{2\pi}{d+5})$, so it follows that 
\[
T_3 \geq \frac{2y(1-\cos(\frac{2\pi}{d+5}))(e^{(2d+4)y}+e^{(d+5)y})}{(e^{(d+3)y}-1)(e^{(d+3)y}+1)^2}.
\]
This bound is decreasing as a function in $y$.  Since $y<2\pi x \leq \pi$, replacing $y$ with $\pi$ gives the following explicit bound on $T_3$ in this case,
\begin{equation}\label{bound1}
T_3 \geq \frac{2\pi(1-\cos(\frac{2\pi}{d+5}))(e^{(2d+4)\pi}+e^{(d+5)\pi})}{(e^{(d+3)\pi}-1)(e^{(d+3)\pi}+1)^2}.
\end{equation}

{\bf Case 2.}  Suppose that $0<\frac{y}{2\pi} < x <\frac{d+4}{2(d+5)}$, and $\frac{\pi}{d+3}\leq y \leq \pi$.  Then $\frac{1}{2(d+3)} < x < \frac{d+4}{2(d+5)}$, and $1-\cos(4\pi x)$ has its minimum in this range at $x=\frac{d+4}{2(d+5)}$.  Since $\cos(\frac{2(d+4)\pi}{d+5}) = \cos (\frac{2\pi}{d+5})$, it follows that
\[
T_1 \geq \frac{y (1-\cos(\frac{2\pi}{d+5}))(e^{(3d+7)y}+e^{2y}-e^{(2d+4)y}-e^{(d+5)y})}{(e^{(d+3)y}-1)(e^{(d+3)y}+1)^2}.
\]
This bound is not decreasing as a function in $y$ in this range, instead it increases and then decreases.  So the mimimum will occur at either $y=\frac{\pi}{d+3}$, or $y=\pi$.  Subtracting the above bound's value at $y=\pi$ from the value at $y=\frac{\pi}{d+3}$ yields an increasing function in $d$ which is positive at $d=4$, thus we conclude that the minimum occurs at $y=\pi$.  Thus, 
\begin{equation} \label{bound2}
T_1 \geq \frac{\pi (1-\cos(\frac{2\pi}{d+5}))(e^{(3d+7)\pi}+e^{2\pi}-e^{(2d+4)\pi}-e^{(d+5)\pi})}{(e^{(d+3)\pi}-1)(e^{(d+3)\pi}+1)^2}.
\end{equation}

{\bf Case 3.}  Suppose that $0<\frac{y}{2\pi} < x <\frac{d+4}{2(d+5)}$, and $0 < y < \frac{\pi}{d+3}$.  Observe that $1-\cos(2(d+3)\pi x)$ is zero exactly when $x=\frac{k}{d+3}$ for some integer $k$.  For fixed $x$, let $k$ denote the integer that minimizes $|x-\frac{k}{d+3}|$, i.e., $\frac{k}{d+3}$ is the zero of $1-\cos(2(d+3)\pi x)$ that is closest to $x$.  Since the zeros are $\frac{1}{d+3}$ apart, it must be that $|x-\frac{k}{d+3}|\leq \frac{1}{2(d+3)}$. 

Suppose that $|x-\frac{k}{d+3}|< \frac{y}{d+3}$.  In this case we can use $T_1$ to obtain a bound. From the periodicity of the cosine function, $1-\cos\frac{4\pi(d+3-\pi)}{(d+3)^2}$ is a lower bound for $1-\cos(4\pi x)$ where $|x-\frac{k}{d+3}|<\frac{\pi}{d+3}$. Thus, 
\begin{equation*}
    T_1\geq \frac{\pi\left(1-\cos\frac{4\pi(d+3-\pi)}{(d+3)^2}\right)\left(e^{(3d+7)y}+e^{2y}-e^{(2d+4)y}-e^{(d+5)y})\right)}{(e^{(d+3)y}-1)(e^{2(d+3)y}-2e^{(d+3)y}\cos(2\pi x(d+3))+1)}.
\end{equation*}
Now, using the $2^{nd}$ order Taylor series expansion of $\cos(2\pi x(d+3))$ about $\frac{k}{d+3}$ and using the error estimate for alternating series we can bound further
\begin{equation*}
    T_1\geq \frac{\pi\left(1-\cos\frac{(d+3)-\pi}{(d+3)^2}\right)\left(e^{(3d+7)y}+e^{2y}-e^{(2d+4)y}-e^{(d+5)y}\right)}{\left(e^{\pi(d+3)}-1\right)\left(\left(e^{(d+3)y}-1\right)^2+8\pi^2y^2e^{(d+3)y}\right)}.
\end{equation*}
As a function of $y$, the above function is increasing, so we take the limit as $y$ approaches $0$ and find
\begin{equation} \label{bound3}
    T_1\geq \frac{\pi\left(1-\cos\frac{(d+3)-\pi}{(d+3)^2}\right)\left(2(d+3)(d+1)\right)}{\left(e^{\pi(d+3)}-1\right)\left(8\pi^2+(d+3)^2\right)}.
\end{equation}

Now suppose that $|x-\frac{k}{d+3}|\geq \frac{y}{d+3}$.  In this case we can use $T_3$ to obtain a bound.  Let $u=2\pi(d+3)|x-\frac{k}{d+3}|$.  Then $2\pi y \leq u$ and for $0<x<\frac{d+4}{2(d+5)}$ we have $0< u \leq \pi$.  Thus, in this range $0 < y \leq \frac12$.  Moreover, $\cos(2\pi(d+3)x)=\cos(u)$.  So, 
\begin{align*}
T_3 & =\frac{ 2y(1-\cos(u))(e^{(2d+4)y}+e^{(d+5)y})}{(e^{(d+3)y}-1)(e^{2(d+3)y}-2e^{(d+3)y}\cos (u)+1)} \\
& \geq \frac{ 4y(1-\cos(u))}{(e^{(d+3)y}-1)((e^{(d+3)y}-1)^2+2(1-\cos (u))e^{(d+3)y})} \\
& \geq \frac{4\pi(1-\cos(u))}{(e^{(d+3)\pi}-1)((e^{\frac{(d+3)u}{2\pi}}-1)^2+2(1-\cos (u))e^{\frac{(d+3)u}{2\pi}})},
\end{align*}
where the last inequality is because $y/(e^{(d+3)y}-1)$ is decreasing in $y$, and the rest is increasing in $y$.  Thus by \cite[eq. (2.3)]{AJLO}, we have
\begin{equation} \label{bound4}
T_3 \geq \frac{8\pi}{(e^{(d+3)\pi}-1)\left( (e^\frac{d+3}{2} -1)^2 + 4e^\frac{d+3}{2} \right)}.
\end{equation}
We conclude by letting $c_2$ be the minimum of the bounds \eqref{bound1}, \eqref{bound2}, \eqref{bound3}, and \eqref{bound4}.
\end{proof}

We need an additional lemma, which is akin to Alfes et al. \cite[Lemma 2.4]{AJLO}, before we can proceed with the proof of Theorem \ref{Q2asymptotics}.

\begin{lemma}\label{lemma:2}
If $\arg(\tau)\leq \frac{\pi}{4}$ and $|x|\leq \frac{1}{2}$, then with $f(\tau)$ defined as in \eqref{def f(tau)}
\begin{equation*}
    f(\tau)=\frac{1}{2\sin(\frac{2\pi}{d+3})}\exp\left(\frac{\pi ^2}{3(d+3)}\tau^{-1}+f_2(\tau )\right),
\end{equation*} where  $|f_2(\tau)|< 0.224\sqrt{y}$.  

Furthermore, if we fix constants $0<\delta<\frac{2}{3}$, $0<\varepsilon_1<\frac{\delta}{2}$, $\beta=\frac{3}{2}-\frac{\delta}{4}$, and require $y^{\beta}\leq |x|\leq \frac{1}{2}$, then when a bound $y_{\text{max}}$ is chosen so that $0<y\leq y_{\text{max}}$ is sufficiently small, there is a constant $c_3$ depending on $d, \varepsilon_1,$ and $\delta$ such that
\begin{equation*}
    f(y+2\pi i x)\leq \exp\left(\frac{\pi ^2}{3(d+3)}y^{-1}-c_3y^{-\varepsilon_1}\right).
\end{equation*}
\end{lemma}

\begin{proof}
Applying \cite[(6.2.7)]{andrews_book} with $A=\frac{2}{d+3}$, $\alpha=1$, $C_0=\frac{1}{2}$, and $D$ as in \eqref{Ddef}, we have by \eqref{eq:D'(0)} that
\begin{equation*}
    \log f(\tau )=\frac{\pi^2}{3(d+3)}\tau^{-1} + \log\bigg(\frac{1}{2\sin\left( \frac{2\pi}{d+3}\right)} \bigg) + \frac{1}{2\pi i}\int_{-\frac{1}{2}-i\infty}^{-\frac{1}{2}+i \infty}\tau ^{-s}\Gamma (s)\zeta (s+1)D(s)ds.
\end{equation*}
Since $|D(s)|\leq |\zeta(s)|$, we obtain directly from \cite[proof of Lemma 2.4]{AJLO} that
\begin{equation*}
\left| \frac{1}{2\pi i}\int_{-\frac{1}{2}-i\infty}^{\frac{1}{2}+i\infty}\tau ^{-s}\Gamma (s)\zeta(s+1)D(s)ds\right|\leq \xi \sqrt{y},
\end{equation*}
where $\xi$ is the constant
\begin{equation} \label{xi}
    \xi=\frac{\sqrt{2}}{2\pi}\int_{-\infty}^{\infty}\left|\zeta\left(\frac{1}{2}+it\right)\zeta\left(-\frac{1}{2}+it\right)\Gamma\left(-\frac{1}{2}+it\right)\right|dt < 0.224.
\end{equation} 
This proves the first statement of the Lemma \ref{lemma:2}. 

To prove the second statement, we consider Case 1 when $y^{\beta}\leq| x|\leq \frac{y}{2\pi}$, and Case 2 when  $\frac{y}{2\pi}< |x|\leq \frac{1}{2}$  separately.
 
In Case 1, we have that $|\text{arg}(\tau)|\leq \frac{\pi}{4},$ so applying the first statement of Lemma \ref{lemma:2} gives
\[
|f(y+2\pi i x)| \leq \left(\frac{1}{2\sin\frac{2\pi}{d+3}}\right)\exp\left(\frac{\pi ^2}{3(d+3)} \frac{y}{y^2 + 4\pi^2x^2} \right)\exp(\xi\sqrt{y}),
\]
and thus,
\begin{multline}\label{logbdd}
    \log|f(y+2\pi i x)| \leq 
    \log\left(\frac{1}{2\sin(\frac{2\pi}{d+3})}\right) + \left( \frac{\pi ^2}{3(d+3)}\cdot \frac{y}{y^2 + 4\pi^2x^2} \right) + \xi\sqrt{y}\\
    =\frac{\pi^2}{3(d+3)}y^{-1} + \frac{\pi^2}{3(d+3)}y^{-1}\left((1+4\pi^2x^2y^{-2})^{-1}-1\right) + \log\left(\frac{1}{2\sin(\frac{2\pi}{d+3})}\right) + \xi\sqrt{y}.
\end{multline}
We note that the final line gives a slightly different bound than in \cite{AJLO}, where the term $(1+4\pi^2x^2y^{-2})^{-\frac{1}{2}}$ appears instead of $(1+4\pi^2x^2y^{-2})^{-1}$. Thus we obtain a slightly better bound here.
Simplifying, we obtain 
\[
\left((1+4\pi^2x^2y^{-2})^{-1}-1\right) = -\frac{4\pi^2x^2y^{-2}}{1+4\pi^2x^2y^{-2}}.
\]
We find an upper bound on this negative term by finding a lower bound for its absolute value.  Observe that since in this case $y^{\beta}\leq |x|\leq \frac{y}{2\pi}$, we have $y^{2\beta-2}\leq x^2y^{-2}\leq  \frac{1}{4\pi^2}$. Hence,
\[
\left((1+4\pi^2x^2y^{-2})^{-1}-1\right) = -\frac{4\pi^2x^2y^{-2}}{1+4\pi^2x^2y^{-2}}  \leq  -2\pi^2 y^{2\beta -2}.
\] 
Thus by \eqref{logbdd}, 
\begin{multline*}
\log|f(y+2\pi i x)| \leq \frac{\pi^2}{3(d+3)}y^{-1} - \frac{2\pi^4}{3(d+3)} y^{2\beta -3} + \log\left(\frac{1}{2\sin(\frac{2\pi}{d+3})}\right) + \xi\sqrt{y} \\
= \frac{\pi^2}{3(d+3)}y^{-1} - y^{-\frac{\delta}{2}}\left( \frac{2\pi^4}{3(d+3)}  -  \log\left(2\sin\left(\frac{2\pi}{d+3}\right)\right)y^{\frac{\delta}{2}}-\xi y^{\frac{1+\delta}{2}} \right),
\end{multline*}
and we have that
\begin{equation*}
    \log |f(y+2\pi ix)|\leq \frac{\pi^2}{3(d+3)}y^{-1}  -  c_4y^{-\frac{\delta}{2}},
\end{equation*}
for the constant $c_4$ defined by
\begin{equation*}
    c_4=\frac{2\pi^4}{3(d+3)} - \log\left(2\sin\left(\frac{2\pi}{d+3}\right)\right) y^{\frac{\delta}{2}}_{\text{max}}-\xi y^{\frac{1+\delta}{2}}_{\text{max}}.
\end{equation*}
Observe that if we choose $y_{\text{max}}$ sufficiently small we can guarantee that $c_4>0$.

In Case 2, where $\frac{y}{2\pi}< |x|\leq \frac{1}{2}$, observe that $\arg(\tau)>\frac{\pi}{4}$ so,  we have as in \cite[(6.2.13)]{andrews_book} that 
\begin{equation}\label{eq log|f|}
    \log |f(y+2\pi i x)|=\log f(y)+Re(g(\tau))-g(y)\leq \frac{\pi^2}{3(d+3)}y^{-1}+Re(g(\tau))-g(y).
\end{equation}

Using Lemma \ref{lemma Ti}, \eqref{eq log|f|} can be bounded as follows:
\begin{equation*}
    \log|f(y+2\pi i x)|\leq \frac{\pi ^2}{3(d+3)}y^{-1}-c_5y^{-1}
\end{equation*}
where $c_5=c_2+y_{\text{max}}\log\left(2\sin\frac{2\pi}{d+3}\right)-\xi y^{\frac{3}{2}}_{\text{max}}$.  Again we observe that by choosing $y_{\text{max}}$ sufficiently small we can guarantee that $c_5>0$.

We require that $y_{\max}$ is small enough to ensure $c_4$ and $c_5$ are positive.  Defining 
\begin{equation*}
    c_3=\text{min}\left(c_4(y_{\text{max}})^{\varepsilon _1-\frac{\delta}{2}}, c_5(y_{\text{max}})^{\varepsilon _1-1}\right),
\end{equation*}
guarantees the inequalities $c_4y^{-\frac{\delta}{2}}\geq c_3y^{-\varepsilon_1}$ and $c_5y^{-1}\geq c_3y^{\varepsilon_1}$.

Thus, for all $x$ such that $y^{\beta}\leq |x|\leq \frac{1}{2}$, 
\begin{equation*}
    \log |f(y+2\pi i x)|\leq \frac{\pi^2}{3(d+3)}y^{-1}-c_3y^{-\varepsilon_1}.
\end{equation*}

\end{proof}

We are now ready to prove Theorem \ref{Q2asymptotics} following the method described in Alfes et al. \cite[Thm. 2.1]{AJLO} and Andrews \cite[Thm. 6.2]{andrews_book}.

\begin{proof}[Proof of Theorem \ref{Q2asymptotics}]
Recall $f(\tau)$ as defined in \eqref{def f(tau)}.  By the Cauchy integral theorem,
\begin{equation*}
    Q_{d}^{(2)}(n)=\frac{1}{2\pi i}\int_{\tau _0}^{\tau _0+2\pi i}f(\tau ) e^{n\tau}d\tau
    =\int_{-\frac{1}{2}}^{\frac{1}{2}}f(y+2\pi i x)e^{ny+2\pi i nx}dx.
\end{equation*} 
We will apply the saddle point method.  Set 
\begin{equation*}
    y=n^{-\frac{1}{\alpha +1}}(A\Gamma (\alpha +1)\zeta (\alpha +1))^{\frac{1}{\alpha +1}}=n^{-\frac{1}{2}}\frac{\pi}{\sqrt{3(d+3)}},
\end{equation*}
where here $\alpha=1$, $A=\frac{2}{d+3}$ and for notational simplicity we define $m=ny$. Fix $0<\delta<\frac{2}{3}$, $0<\varepsilon_1<\frac{\delta}{2}$, and $\beta=\frac{3}{2}-\frac{\delta}{4}$.  As in the proof of \cite[Thm. 2.1]{AJLO}, we assume $n\geq 6$ which guarantees that $y \leq \left(\frac{1}{2\pi}\right)^{\frac{1}{\beta-1}}$. This implies $y^{\beta}\leq \frac{y}{2\pi}$ and so $\frac{y}{2\pi}\leq \frac{1}{2}$. Thus both intervals in Cases 1 and 2 in the proof of the second statement of Lemma \ref{lemma:2} are nonempty, so we have 
\begin{equation} \label{Qd full integral}
    Q_{d}^{(2)}(n)=e^m\int_{-y^{\beta}}^{y^{\beta}}f(y+2\pi i x)\exp (2\pi i n x)dx +e^m R_1,
\end{equation}

\noindent where \begin{equation*}
    R_1=\left(\int_{-\frac{1}{2}}^{-y^{\beta}}+\int_{y^{\beta}}^{\frac{1}{2}}\right)f(y+2\pi i x)\exp(2\pi i n x)dx.
\end{equation*}
Then, since $y^{\beta}\leq |x| \leq \frac{1}{2}$ in the integrals defining $R_1$, we use Lemma \ref{lemma:2}
to obtain the bound
\begin{equation}\label{R1bound}
   |R_1|\leq \exp \left[\frac{\pi^2}{3(d+3)}\frac{n}{m}-c_3\left(\frac{n}{m}\right)^{\varepsilon_1} \right].
\end{equation}

\noindent Multiplying both sides by $e^m$ gives
\begin{equation}
    |e^mR_1|\leq\exp\left[2m-c_3m^{\varepsilon_1}\left(\frac{\pi ^2}{3(d+3)}\right)^{-\varepsilon_1}\right].
\end{equation}

Now, we turn our attention to the first integral of \eqref{Qd full integral}, 
\begin{equation}\label{Interior integral}
e^m\int_{-y^{\beta}}^{y^{\beta}}f(y+2\pi i x)\exp (2\pi i n x)dx.
\end{equation}

By Lemma \ref{lemma:2} and \cite[(6.2.21), (6.2.22)]{andrews_book}, we obtain
\begin{equation*}
    Q_{d}^{(2)}(n)=\exp\left(2m+\log\left(\frac{1}{2\sin\frac{2\pi}{d+3}}\right)\right)\int_{-(m/n)^{\beta}}^{(m/n)^{\beta}}\exp(\varphi _1(x))dx+\exp(m)R_1,
\end{equation*}
where
\begin{multline*}
    \varphi_1(x)= m\left[\left(1+\frac{2\pi i x n}{m}\right)^{-1}-1\right]+2\pi i n x -D(0)\log\left(\frac{m}{n}+2\pi i x\right)+g_1(x)\\
    =m\left[\left(1+\frac{2\pi i x n}{m}\right)^{-1}-1\right]+2\pi i nx +g_1(x),
\end{multline*}
and for a constant $\xi$,
\begin{equation}
    |g_1(x)|\leq \xi m^{-\frac12} \frac{\pi}{\sqrt{3(d+3)}}.
\end{equation}

Making the change of variables $2\pi x =(m/n)\omega$, it follows that
\begin{multline} \label{Qrevised}
    Q_{d}^{(2)}(n)=\exp\left(2m+\log\frac{m}{n} + \log\left(\frac{1}{2\sin\frac{2\pi}{d+3}}\right)-\log 2\pi\right)I + e^m R_1 \\
    =  \frac{m}{4\pi n \sin(\frac{2\pi}{d+3})} e^{2m}I+e^mR_1,
\end{multline}
where
\begin{align*}
    I & = \int_{-c_{10}m^{1-\beta}}^{c_{10}m^{1-\beta}}\exp(\varphi_2(\omega))d\omega, \\
    c_{10} & = 2\pi\left(\frac{\pi^2}{3(d+3)}\right)^{\beta-1}, \\
    \varphi_2(\omega) & = m\left(\frac{1}{1+i\omega}-1+i\omega\right)+g_1\left(\frac{m\omega}{2\pi n}\right),
\end{align*}
and $R_1$ is bounded as in \eqref{R1bound}.  

Now, rather than finding an asymptotic expression for \eqref{Interior integral}, we instead find an asymptotic expression for $I$.  Write
\begin{equation}\label{I=int+R2}
    I=\int_{-c_{10}m^{1-\beta}}^{c_{10}m^{1-\beta}}\exp(-m\omega^2)d\omega +R_2,
\end{equation}
where
\begin{equation*}
    R_2=\int_{-c_{10}m^{1-\beta}}^{c_{10}m^{1-\beta}}\exp(-m\omega^2)(\exp(\varphi_3(\omega))-1)d\omega,
\end{equation*}
with
\begin{equation*}
    \varphi_3(\omega)=m\left(\frac{1}{1+i\omega}-1+i\omega+\omega ^2\right)+g_1(\omega )
    =m\left(\frac{i\omega^3}{1+i\omega}\right)+g_1(\omega).
\end{equation*}
Thus, we can bound $\varphi_3(\omega)$ on the interval $[-c_{10}m^{1-\beta}, c_{10}m^{1-\beta}]$ by
\begin{equation}\label{phi3bound}
    |\varphi_3(\omega)|\leq c_{10}^3 m^{4-3\beta}+ \xi m^{-\frac12} \frac{\pi}{\sqrt{3(d+3)}}= c_{10}^3 m^{\frac{3\delta-2}{4}}+ \xi m^{-\frac12} \frac{\pi}{\sqrt{3(d+3)}}.
\end{equation}
Since $\frac{3\delta-2}{4}$ is negative, minimizing $m$ will yield an upper bound.  Thus 
\begin{equation} \label{minm}
m=ny=\frac{\pi\sqrt{n}}{\sqrt{3(d+3)}}\geq \frac{\sqrt{2}\pi}{3\sqrt{d+3}}   
\end{equation}
implies
\begin{equation*}
    |\varphi_3(\omega)|\leq\frac{2^{\frac{22+3\delta}{8}}\pi^{\frac{22-3\delta}{4}}}{3(d+3)^{\frac{10-3\delta}{8}}}+\xi \left( \frac{\pi^2}{2(d+3)}\right)^{\frac{1}{4}} =:\varphi_{3,\text{max}}.
\end{equation*}
Define the constant $c_6$ by $c_6:=\frac{\exp(\varphi_{3,\text{max}})-1}{\varphi_{3,\text{max}}}$.  Then using \eqref{phi3bound} and \eqref{minm}, 
\begin{multline}\label{integrandbound}
|\exp(\varphi_{3}(\omega)-1| = |\varphi_3(\omega)|c_6 \leq\left( c_{10}^3m^{\frac{3\delta-2}{4}}+\xi\sqrt{\frac{\pi^2}{3m(d+3)}}\right)c_6\\
=m^{\frac{3\delta-2}{4}}\left(c_6c_{10}^3+\xi c_6m^{-\frac{3\delta}{4}}\sqrt{\frac{\pi^2}{3(d+3)}}\right) 
\leq m^{\frac{3\delta-2}{4}}\left( c_6c_{10}^3+\xi c_6 \frac{\pi^{\frac{4-3\delta}{4}}}{2^{\frac{3\delta}{8}}3^{\frac{2-3\delta}{4}}(d+3)^{\frac{4-3\delta}{8}}}\right) :=m^{\frac{3\delta-2}{4}}c_7.
\end{multline}
Using \eqref{integrandbound} we thus obtain
\begin{equation}\label{R2bound}
    |R_2|\leq\int_{-c_{10}m^{1-\beta}}^{c_{10}m^{1-\beta}}\exp(-m\omega^2)m^{\frac{3\delta-2}{4}}c_7d\omega\leq 2c_{10}c_7m^{\delta-1}.
\end{equation}

Returning to $I$, we now have by \eqref{I=int+R2} and the change of variable  $z=m^{\frac12}\omega$, that
\[
I = \frac{1}{\sqrt{m}}\int_{-c_{10}m^{\frac{\delta}{4}}}^{c_{10}m^{\frac{\delta}{4}}}\exp(-z^2)dz + R_2 =\frac{1}{\sqrt{m}}\int_{-\infty}^{\infty}\exp(-z^2)dz-\frac{2}{\sqrt{m}}\int_{c_{10}m^{\frac{\delta}{4}}}^{\infty}\exp (-z^2)dz + R_2,
\]
with $R_2$ bounded as in \eqref{R2bound}.  Defining 
 \begin{equation}\label{negative g_2}
     g_2(z):=-\frac{2}{\sqrt{m}}\int_{c_{10}m^{\frac{\delta}{4}}}^{\infty}\exp (-z^2)dz,
 \end{equation}
we write $I$ as 
\begin{equation}\label{Irevised}
I=\left(\frac{\pi}{m}\right)^{\frac{1}{2}}+g_2(m)+R_2,
\end{equation}
where $g_2(m)$ is negative and $|g_2(m)|\leq\frac{2}{\sqrt{m}}\exp\left(-c_{10}m^{\frac{\delta}{4}}\right)$.
Thus from \eqref{Qrevised} and \eqref{Irevised} we now have that
\begin{multline*}
Q_{d}^{(2)}(n) = \frac{m}{4\pi n \sin(\frac{2\pi}{d+3})} e^{2m}\left( \left(\frac{\pi}{m}\right)^{\frac{1}{2}}+g_2(m)+R_2 \right)+e^mR_1 = \frac{e^{2m}\sqrt{m}}{4n\sqrt{\pi}\sin(\frac{2\pi}{d+3})} + R_d(n),
\end{multline*}
where 
\[
R_d(n) = \frac{m e^{2m}}{4\pi n \sin(\frac{2\pi}{d+3})} \left(g_2(m) + R_2 \right) + e^mR_1.
\]
Writing this in terms of the variable $n$ and using \eqref{R1bound} and \eqref{R2bound}, we obtain our desired result.  Namely,
\[
Q_{d}^{(2)}(n)= \frac{1}{4(3(d+3))^{\frac{1}{4}}\sin\left(\frac{2\pi}{d+3}\right)}n^{-\frac{3}{4}}\exp\left(\frac{2\pi \sqrt{n}}{\sqrt{3(d+3)}}\right) + R_d(n),
\]
where
\begin{multline}\label{|R_d(n)|}
    |R_d(n)|\leq n^{-\frac{1}{4}}\left(\frac{\pi^{\frac{1}{2}}(3(d+3))^{-\frac{3}{4}}}{2\sin(\frac{2\pi}{d+3})}\right)\exp\left(\frac{2\pi\sqrt{n}}{\sqrt{3(d+3)}}-n^{-\frac{\delta}{8}}2\pi^{2-\frac{\delta}{4}}(3(d+3))^{-2+\frac{3\delta}{8}}\right)\\
    + n^{-1+\frac{\delta}{2}}\left(\frac{c_7\pi^{1+\frac{\delta}{2}}}{(3(d+3))^2\sin\left(\frac{2\pi}{d+3}\right)}\right)\exp\left(\frac{2\pi\sqrt{n}}{\sqrt{3(d+3)}}\right)\\
    +\exp\left(\frac{2\pi\sqrt{n}}{\sqrt{3(d+3)}}-c_3n^{\frac{\varepsilon_1}{2}}\left(\frac{\pi^2}{3(d+3)}\right)^{-\frac{3\varepsilon_1}{2}}\right).
\end{multline}
\end{proof}

\subsection{An upper bound on $Q_{d}^{(2)}(n)$}

From the proof of Theorem \ref{Q2asymptotics}, since $g_2(m)$ is negative, we observe that 
\[
0\leq Q_{d}^{(2)}(n) \leq \frac{1}{4(3(d+3))^{\frac{1}{4}}\sin\left(\frac{2\pi}{d+3}\right)}n^{-\frac{3}{4}}\exp\left(\frac{2\pi \sqrt{n}}{\sqrt{3(d+3)}}\right) + \frac{m e^{2m}}{4\pi n \sin(\frac{2\pi}{d+3})} R_2 + e^mR_1.
\]
Moreover, $c_3>0$, so \eqref{R1bound} and \eqref{R2bound} give that
\begin{multline}\label{eq:Q2bound}
|Q_{d}^{(2)}(n)| \leq \frac{1}{4(3(d+3))^{\frac{1}{4}}\sin\left(\frac{2\pi}{d+3}\right)}n^{-\frac{3}{4}}\exp\left(\frac{2\pi \sqrt{n}}{\sqrt{3(d+3)}}\right) \\ + n^{-1+\frac{\delta}{2}}\left(\frac{c_7\pi^{1+\frac{\delta}{2}}}{(3(d+3))^2\sin\left(\frac{2\pi}{d+3}\right)}\right)\exp\left(\frac{2\pi\sqrt{n}}{\sqrt{3(d+3)}}\right)
    +\exp\left(\frac{2\pi\sqrt{n}}{\sqrt{3(d+3)}}\right). 
\end{multline}
Similarly, from the proof of \cite[Thm. 2.1]{AJLO}, 
\begin{multline}\label{eq:Q1bound}
|Q_{d}^{(1)}(n)| \leq \frac{1}{4(3(d+3))^{\frac{1}{4}}\sin\left(\frac{\pi}{d+3}\right)}n^{-\frac{3}{4}}\exp\left(\frac{2\pi \sqrt{n}}{\sqrt{3(d+3)}}\right) \\
+ n^{-1+\frac{\delta}{2}}\left(\frac{c_7\pi^{1+\frac{\delta}{2}}}{(3(d+3))^2\sin(\frac{\pi}{d+3})}\right)\exp\left(\frac{2\pi\sqrt{n}}{\sqrt{3(d+3)}}\right) +\exp\left(\frac{2\pi\sqrt{n}}{\sqrt{3(d+3)}}\right),
\end{multline}
where in \eqref{eq:Q1bound} the positive real constants $\delta$, $c_7$, $c_3$, and $\varepsilon_1$ are defined separately, but analogously, as in \cite{AJLO}.

Due to the parallel nature of these bounds, we can combine them into one expression.  Namely, for $b\in \{1,2\}$ and $d\geq 4$ we have
\begin{multline}\label{combinedQ}
|Q_d^{(b)}(n)| \leq \frac{1}{4(3(d+3))^{\frac{1}{4}}\sin\left(\frac{b\pi}{d+3}\right)}n^{-\frac{3}{4}}\exp\left(\frac{2\pi\sqrt{n}}{\sqrt{3(d+3)}}\right) \\
+ n^{-1+\frac{\delta}{2}}\left(\frac{c_7\pi^{1+\frac{\delta}{2}}}{(3(d+3))^2\sin\left(\frac{b\pi}{d+3}\right)}\right)\exp\left(\frac{2\pi\sqrt{n}}{\sqrt{3(d+3)}}\right) +\exp\left(\frac{2\pi\sqrt{n}}{\sqrt{3(d+3)}}\right),
\end{multline}
where again the positive real constants $\delta$, $c_7$, $c_3$, and $\varepsilon_1$ are defined separately, but analogously, depending on $b$.  They are defined in Section \ref{asymptotics} when $b=2$ and as in \cite{AJLO} when $b=1$. 
 We will use \eqref{combinedQ} in Section \ref{bounds}.

\section{Obtaining explicit bounds and Proof of Kang-Park}\label{bounds}

Recall that our goal is to determine positive integers $N(d)$ for each $4\leq d \leq 61$ such that when $d$ is even we have $q_d^{(2)}(n) \geq Q_d^{(2)}(n)$, and when $d$ is odd, we have $q_d^{(2)}(n) \geq Q_d^{(1)}(n)$ for all $n\geq N(d)$.   From Theorem \ref{q2asymptotics}, the main term of $q_d^{(2)}(n)$ is
\begin{equation*}
m_d(n) := \frac{A_d^{1/4}}{2\sqrt{\pi\alpha^{d-3}(d\alpha^{d-1}+1)}}n^{-\frac{3}{4}}\exp\left(2\sqrt{A_d n}\right).
\end{equation*}
Thus we need to compare $m_d(n)$ with the sum of the bound for $Q_d^{(b)}(n)$ given in \eqref{combinedQ} and the bounds for $r(n)$ given in \eqref{E1'(n)}, \eqref{E2(n)}, \eqref{E3(n)}, and \eqref{I_2}.
 Namely, we need to determine $N(d)$ such that for all $n\geq N(d)$,
\[
|Q_d^{(b)}(n)| + |r_d(n)| \leq m_d(n).
\]

Since the bounds for $|Q_d^{(b)}(n)|$ and $|r_d(n)|$ are sums, we approach this by writing $$|Q_d^{(b)}(n)| + |r_d(n)| = \sum_{i=1}^8 S_i,$$ and finding $N_i$ depending on a weight $K_i$ such that for $n\geq N_i$ we have $S_i\leq K_i m_d(n)$ in each case.  Then choosing $K_i$ so that $\sum_{i=1}^8 K_i=1$ and setting $N(d) = \max\{N_i\}$ gives that for all $n\geq N(d)$,
\[
|Q_d^{(b)}(n)| + |r_d(n)| = \sum_{i=1}^8 S_i \leq \sum_{i=1}^8 K_im_d(n) = m_d(n),
\]
which ensures that $q_d^{(2)}(n)\geq Q_d^{(b)}(n)$ for all $n\geq N(d)$.  We accomplish this in the following two lemmas, the first addressing $|Q_d^{(b)}(n)|$ and the second $|r_d(n)|$. 

\begin{lemma}\label{Qdbounds}
   Let $4\leq d\leq 61$, and $b\in\{1,2\}$. Let $\alpha$ be the unique real solution of $x^d+x-1=0$ in the interval $(0,1)$, and let $A_d:=\frac{d}{2}\log ^2\alpha+\sum_{r=1}^{\infty}\frac{\alpha^{rd}}{r^2}$.  Fix weights $K_1, K_2, K_3\in(0,1)$. 
 Then there exists an explicit positive integer $N_Q$ depending on $K_i$ (defined in \eqref{N_Qdef}) such that for all $n\geq N_Q$, 
   \begin{equation*}
       Q_d^{(b)}(n)\leq (K_1+K_2+K_3)\frac{A_d^{1/4}}{2\sqrt{\pi\alpha^{d-3}(d\alpha^{d-1}+1)}}n^{-\frac{3}{4}}\exp\left(2\sqrt{A_dn}\right).
   \end{equation*}
\end{lemma}
\begin{proof}
For $i\in \{1,2,3\}$, let $S_i$ denote the $i$th summand appearing in the right hand side of \eqref{combinedQ}, so that $Q_d^{(b)}(n)\leq \sum_{i=1}^3 S_i$.  The inequality $S_1 \leq K_1m_d(n)$ is equivalent to the following

    \begin{equation*}
        \log\left(\frac{\sqrt{\pi\alpha^{d-3}(d\alpha^{d+1}+1)}}{2K_1\sin\frac{b\pi}{d+3}(3(d+3)A_d)^{\frac{1}{4}}}\right)\leq \sqrt{n}\left(2\sqrt{A_d}-\frac{2\pi}{\sqrt{3(d+3)}}\right), 
    \end{equation*}
    or more directly,
    \begin{equation*}
        n\geq\left(\log\left(\frac{\sqrt{\pi\alpha^{d-3}(d\alpha^{d+1}+1)}}{2K_1\sin\frac{b\pi}{d+3}(3(d+3)A_d)^{\frac{1}{4}}}\right)\left(2\sqrt{A_d}-\frac{2\pi}{\sqrt{3(d+3)}}\right)^{-1}\right)^2.
    \end{equation*}
    Thus, defining
    \begin{equation}\label{N1}
        N_1:=\Bigg\lceil \left(\log\left(\frac{\sqrt{\pi\alpha^{d-3}(d\alpha^{d-1}+1)}}{2K_1\sin\frac{b\pi}{d+3}(3(d+3)A_d)^{\frac14}}\right)\right)^2\left(2\sqrt{A_d}-\frac{2\pi}{\sqrt{3(d+3)}}\right)^{-2}\Bigg\rceil,
    \end{equation}
ensures that $S_1 \leq K_1m_d(n)$ for all $n\geq N_1$.

Since $0<\delta<\frac{1}{2}$, to obtain $S_2 \leq K_2m_d(n)$ it suffices to show
\[
n^{-\frac{3}{4}}\left(\frac{c_7\pi^{1+\delta/2}}{(3(d+3))^2\sin\frac{b\pi}{d+3}}\right)\exp\left(\frac{2\pi\sqrt{n}}{\sqrt{3(d+3)}}\right)\leq\frac{K_2A_d^{1/4}}{2\sqrt{\pi\alpha^{d-3}(d\alpha^{d-1}+1)}}n^{-\frac{3}{4}}e^{2\sqrt{A_dn}},
\]
which is equivalent to 
\[
   \log \left(\left(\frac{c_7\pi^{1+\delta/2}}{(3(d+3))^2\sin\frac{b\pi}{d+3}}\right)\left(\frac{K_2A_d^{1/4}}{2\sqrt{\pi\alpha^{d-3}(d\alpha^{d-1}+1)}}n^{-\frac{3}{4}}\right)^{-1}\right)\leq \sqrt{n}\left(2\sqrt{A_d}-\frac{2\pi}{\sqrt{3(d+3)}}\right), 
\]
or more directly,
\[
    n\geq \left(\log\left(\frac{2c_7\pi^{1+\delta/2}\sqrt{\pi\alpha^{d-3}(d\alpha^{d-1}+1)}}{K_2A_d^{1/4}(3(d+3))^2\sin\frac{b\pi}{d+3}}\right)\right)^2\left(2\sqrt{A_d}-\frac{2\pi}{\sqrt{3(d+3)}}\right)^{-2}.
\]
    
Thus, defining
\begin{equation}\label{N2}
N_2:=  \Bigg\lceil\left(\log\left(\frac{2c_7\pi^{1+\delta/2}\sqrt{\pi\alpha^{d-3}(d\alpha^{d-1}+1)}}{K_2A_d^{1/4}(3(d+3))^2\sin\frac{b\pi}{d+3}}\right)\right)^2\left(2\sqrt{A_d}-\frac{2\pi}{\sqrt{3(d+3)}}\right)^{-2}\Bigg\rceil
\end{equation}
ensures that $S_2 \leq K_2m_d(n)$ for all $n\geq N_2$.  

Lastly, $S_3 \leq K_3m_d(n)$ is equivalent to 
\begin{equation}\label{N3 inequality 2}
\frac{2\sqrt{\pi\alpha^{d-3}(d\alpha^{d-1}+1)}}{K_3A_d^{1/4}}n^{\frac{3}{4}}\leq \exp\left(\left(2\sqrt{A_d}-\frac{2\pi}{\sqrt{3(d+3)}}\right)n^{\frac{1}{2}}\right).
\end{equation}
The equation
\[
\frac{2\sqrt{\pi\alpha^{d-3}(d\alpha^{d-1}+1)}}{K_3A_d^{1/4}}n^{\frac{3}{4}}= \exp\left(\left(2\sqrt{A_d}-\frac{2\pi}{\sqrt{3(d+3)}}\right)n^{\frac{1}{2}}\right)
\]
has two positive solutions $\sigma_0 < \sigma_1$ for $n$ and \eqref{N3 inequality 2} is satisfied for all $n\geq \sigma_1$. We thus define $N_3 := \lceil \sigma_1 \rceil$, which we calculate using a root finding program in Sagemath for each $4\leq d \leq 61$.  Thus we have $S_3 \leq K_3m_d(n)$ for all $n\geq N_3$.

Setting 
\begin{equation}\label{N_Qdef}
N_Q := \max\{N_1,N_2,N_3\}
\end{equation}
gives the desired result.
\end{proof}

\begin{lemma}\label{qdbounds} 
Let $4\leq d\leq 61$ and let $r_d^{(2)}(n)$ be as defined in \cite[Theorem 6.5]{DKT_proc} with $a=2.$    Let $\alpha$ be the unique real solution of $x^d+x-1=0$ in the interval $(0,1)$, and let $A_d:=\frac{d}{2}\log ^2\alpha+\sum_{r=1}^{\infty}\frac{\alpha^{rd}}{r^2}$. Fix weights $K_4, K_5, K_6,K_7,K_8\in (0,1).$  Then, there exists an explicit positive integer $N_q$ depending on $K_i$ (defined in \eqref{N_qdef}) such that for all $n\geq N_q$,

    \begin{equation*}
        r_d^{(2)}(n)\leq (K_4+K_5+K_6+K_7+K_8)\frac{A_d^{1/4}}{2\sqrt{\pi\alpha^{d-3}(d\alpha^{d-1}+1)}}n^{-\frac{3}{4}}\exp\left(2\sqrt{A_dn}\right).
    \end{equation*}
\end{lemma}

\begin{proof}   In order to satisfy the hypotheses of Lemma \ref{q2preliminaries}, choose $\rho=\alpha^d=1-\alpha,\ 0<\xi<1,\ 0<\varepsilon<\frac{1}{2}$, as well as $x=\sqrt{\frac{A_d}{n}}$ and $\gamma:=\frac{1}{2\pi\sqrt{\alpha^{d-3}(d\alpha^{d-1}+1)}}$.  Then, as described in \cite{DKT_proc}, it follows that $r_d^{(2)}(n)=E_1'+E_2+E_3+I_2$, where the summands have explicitly given bounds.  We state these in \eqref{E1'(n)}, \eqref{E2(n)}, \eqref{E3(n)}, and \eqref{I_2}, respectively.
To begin, 
    \begin{equation}\label{E1'(n)}
       | E_1'|\leq S_4:= \frac{\gamma}{\sqrt{2A_d^{\varepsilon}}}n^{\frac{\varepsilon}{2}-1}e^{2\sqrt{2A_dn}-n^{\frac{1}{2}-\varepsilon}A_d^{\frac{1}{2}+\varepsilon}}.
    \end{equation} The inequality $S_4\leq K_4 m_q(d)$ is equivalent to
    \begin{equation}
        \frac{\gamma}{\sqrt{2A_d^{\varepsilon}}}n^{\frac{\varepsilon}{2}-1}e^{2\sqrt{2A_dn}-n^{\frac{1}{2}-\varepsilon}A_d^{\frac{1}{2}+\varepsilon}}\leq K_4\frac{A_d^{1/4}}{2\sqrt{\pi\alpha^{d-3}(d\alpha^{d-1}+1)}}n^{-\frac{3}{4}}\exp\left(2\sqrt{A_dn}\right).
    \end{equation}
    Since $0<\varepsilon<\frac{1}{2}$, it suffices to determine $N_4\in \mathbb{N}$ to ensure that for $n\geq N_4$,
 \begin{equation*}
     \frac{\gamma}{\sqrt{2A_d^{\varepsilon}}}n^{-\frac{3}{4}}e^{2\sqrt{2A_dn}-n^{\frac{1}{2}-\varepsilon}A_d^{\frac{1}{2}+\varepsilon}}\leq K_4\frac{A_d^{1/4}}{2\sqrt{\pi\alpha^{d-3}(d\alpha^{d-1}+1)}}n^{-\frac{3}{4}}\exp\left(2\sqrt{A_dn}\right).
 \end{equation*}
 Equivalently, we have 
 \begin{equation*}
     \log\left(\frac{1}{K_4\sqrt{2\pi}A^{\varepsilon/2+1/4}}\right)\leq n^{\frac{1}{2}-\varepsilon}A_d^{\frac{1}{2}+\varepsilon},
 \end{equation*}
 which implies 
 \begin{equation}\label{N4}
     N_4:=\bigg\lceil\left(A_d^{-\frac{1}{2}+\varepsilon}\log\left(\frac{1}{K_4\sqrt{2\pi}A_d^{\varepsilon/2+1/4}}\right)\right)^{\frac{2}{1-2\varepsilon}}\bigg\rceil.
 \end{equation}
 
    
Let $\varepsilon_2>\frac{1}{3}$ and $\varepsilon_2>\varepsilon$.  Then as in \cite{DKT_proc},
\begin{equation}\label{E2(n)}
 |E_2|\leq S_5 + S_6,   
\end{equation}
where 
\begin{align*}
S_5 := & \gamma e^{2\sqrt{A_d n}}\left(\exp(A_d^{\frac{1}{2}+\frac{3\varepsilon_2}{2}})-1\right)\sqrt{\pi}A_d^{\frac{1}{4}}n^{-\frac{3}{4}}, \\
S_6 := & \gamma\exp\left(2\sqrt{A_dn}-\frac{A_d^{\varepsilon_2/2}n^{1-\varepsilon_2/2}}{1+A_d^{\varepsilon}n^{-\varepsilon}}\right)A_d^{\frac{3}{2}}n^{-\frac{3}{2}}(1+A_d^{\varepsilon}n^{-\varepsilon})\\
&+\gamma A_d^{\frac{1}{2}}n^{-\frac{3}{2}}\exp\left(2\sqrt{A_dn}-A_d^{\varepsilon_2/2}n^{1-\varepsilon_2/2}\right).
\end{align*}
    To obtain $S_5\leq K_5 m_q(d)$ we find $N_5$ to ensure that for $n\geq N_5$,
    \begin{equation}
        \gamma e^{2\sqrt{A_d n}}\left(\exp(A_d^{\frac{1}{2}+\frac{3\varepsilon_2}{2}})-1\right)\sqrt{\pi}A_d^{\frac{1}{4}}n^{-\frac{3}{4}}\leq K_5\frac{A_d^{1/4}}{2\sqrt{\pi\alpha^{d-3}(d\alpha^{d-1}+1)}}n^{-\frac{3}{4}}\exp\left(2\sqrt{A_dn}\right).
    \end{equation}
    Equivalently, \begin{equation*}
        \exp\left(A^{\frac{1+3\varepsilon_2}{2}}_d n^{\frac{1-3\varepsilon_2}{2}}\right)-1\leq K_5,
    \end{equation*}
    or more directly
    \begin{equation*}
        A_d^{\frac{1+3\varepsilon_2}{2}}n^{\frac{1-3\varepsilon_2}{2}}\leq \log(K_5+1).
    \end{equation*}
    Since $\varepsilon_2>\frac{1}{3}$ it suffices to define
    \begin{equation}\label{N5}
        N_5:=\bigg\lceil \left(\frac{A_d^{\frac{1+3\varepsilon_2}{2}}}{\log(K_5+1)}\right)^{\frac{2}{3\varepsilon_2-1}}\bigg\rceil.
    \end{equation}

We next determine $N_6$ such that $S_6\leq K_6m_d(n)$ for all $n\geq N_6$ via the following inequality

   \begin{multline*}
       \gamma\exp\left(2\sqrt{A_dn}-\frac{A_d^{\varepsilon_2/2}n^{1-\varepsilon_2/2}}{1+A_d^{\varepsilon}n^{-\varepsilon}}\right)A_d^{\frac{3}{2}}n^{-\frac{3}{2}}(1+A_d^{\varepsilon}n^{-\varepsilon})+\gamma A_d^{\frac{1}{2}}n^{-\frac{3}{2}}\exp\left(2\sqrt{A_dn}-A_d^{\varepsilon_2/2}n^{1-\varepsilon_2/2}\right)\\ \leq K_6 \frac{A_d^{1/4}}{2\sqrt{\pi\alpha^{d-3}(d\alpha^{d-1}+1)}}n^{-\frac{3}{4}}\exp\left(2\sqrt{A_dn}\right),
   \end{multline*} 
   which is equivalent to 

   \begin{equation*}
       n^{-\frac{3}{4}}\left(\frac{A_d^{5/4}}{\sqrt{\pi}}\exp\left(-\frac{A_d^{\varepsilon_2/2}n^{1-\varepsilon_2/2}}{1+A_d^{\varepsilon}n^{-\varepsilon}}\right)(1+A_d^{\varepsilon}n^{-\varepsilon})+\frac{A_d^{\frac{1}{4}}}{\sqrt{\pi}}\exp(-A_d^{\frac{\varepsilon_2}{2}}n^{1-\frac{\varepsilon_2}{2}})\right)\leq K_6.
   \end{equation*}
   Using the fact that $n\geq 1$ we bound the exponential terms above by 1 and define
   \begin{equation}\label{N6}
       N_6:=\bigg\lceil\left(\frac{A_d^{5/4}(1+A_d^{\varepsilon})+A_d^{\frac{1}{4}}}{\sqrt{\pi}K_6}\right)^{\frac{4}{3}}\bigg\rceil.
   \end{equation}

As in \cite{DKT_proc}, 
\begin{equation}\label{E3(n)}
    |E_3|\leq S_7:=\gamma e^{2\sqrt{A_dn}}|f_{err}^{max}|(\pi A^{3/2}n^{-3/2}(1+A_d^{\varepsilon}n^{-\varepsilon}))^{\frac{1}{2}},
\end{equation}
where $|f_{err}^{max}|$ is a computable constant defined as the maximum value of \cite[(26)]{DKT_proc}.  To obtain $S_7\leq K_7 m_q(d)$ we find $N_7$ to ensure that for $n\geq N_7$,

\begin{equation*}
    \gamma e^{2\sqrt{A_dn}}|f_{err}^{max}|(\pi A^{3/2}n^{-3/2}(1+A_d^{\varepsilon}n^{-\varepsilon}))^{\frac{1}{2}}\leq K_7\frac{A_d^{1/4}}{2\sqrt{\pi\alpha^{d-3}(d\alpha^{d-1}+1)}}n^{-\frac{3}{4}}\exp\left(2\sqrt{A_dn}\right).
\end{equation*}
Simplifying the inequality, we have

\begin{equation*}
    |f_{err}^{\max}|A_d^{\frac{1}{2}}(1+A_d^{\varepsilon}n^{-\varepsilon})^{\frac{1}{2}}\leq K_7.
\end{equation*}
Thus, we can define 
\begin{equation}\label{N7}
    N_7:=\bigg\lceil\left(\frac{A_d^{1+\frac{1}{\varepsilon}}|f_{err}^{\max}|^{\frac{2}{\varepsilon}}}{(K_7^2-|f_{err}^{\max}|^2A_d)^{\frac{1}{\varepsilon}}}\right)\bigg\rceil.
\end{equation}

Set $\beta$ as in Lemma \ref{q2preliminaries}, and define $\eta$ as in \cite{DKT_proc}\footnote{The definition of $\eta$ here corrects some small typos in \cite{DKT_proc}.} by  
\[
\eta := e^{-\beta}\beta^{1-2\varepsilon}e^{-2\beta}\left(\frac{1}{1-e^{\beta}}-\frac{1}{\sqrt{1-2e^{-\beta}\cos (\beta^{1+\varepsilon})+e^{-2\beta}}}\right).
\]
Then as in \cite{DKT_proc},
\begin{multline}\label{I_2}
|I_2|\leq S_8:= \\
\sqrt{\frac{2\pi}{dA_d^{1/2}}}n^{\frac{1}{4}}e^{\eta\rho A_d^{\varepsilon-\frac{1}{2}}n^{\frac{1}{2}-\varepsilon}}(1+f_2(\rho,A_d^{\frac{1}{2}}n^{-\frac{1}{2}}))\exp\left(2\sqrt{A_dn}+\left(\frac{3-d}{2}\right)\log(\alpha)+f_1(\rho,A_d^{\frac{1}{2}}n^{-\frac{1}{2}})\right).
\end{multline}
  Thus, bounding $f_2(\rho, A_d^{\frac{1}{2}}n^{-\frac{1}{2}})\leq F_2$ and $f_1(\rho, A_d^{\frac{1}{2}}n^{-\frac{1}{2}})\leq F_1$ using Lemma \ref{q2preliminaries}, we have  

\begin{equation*}
   S_8\leq \sqrt{\frac{2\pi}{dA_d^{1/2}}}n^{\frac{1}{4}}e^{\eta\rho A_d^{\varepsilon-\frac{1}{2}}n^{\frac{1}{2}-\varepsilon}}(1+F_2)\exp\left(2\sqrt{A_dn}+\left(\frac{3-d}{2}\right)\log(\alpha)+F_1\right).
\end{equation*}
To obtain $S_8\leq K_8 m_q(d)$ we find $N_8$ to ensure that for $n\geq N_8$,
\begin{multline*}
    \sqrt{\frac{2\pi}{dA_d^{1/2}}}n^{\frac{1}{4}}e^{\eta\rho A_d^{\varepsilon-\frac{1}{2}}n^{\frac{1}{2}-\varepsilon}}(1+F_2)\exp\left(2\sqrt{A_dn}+\left(\frac{3-d}{2}\right)\log(\alpha)+F_1\right)\\ \leq K_8\frac{A_d^{1/4}}{2\sqrt{\pi\alpha^{d-3}(d\alpha^{d-1}+1)}}n^{-\frac{3}{4}}\exp\left(2\sqrt{A_dn}\right),
\end{multline*}
or equivalently,
\begin{equation}\label{N8}
    \frac{\sqrt{2}}{K_8\gamma A_d^{\frac{1}{4}}\sqrt{d}}(1+F_2)\exp\left(\frac{3-d}{2}\log\alpha +F_1\right)n\leq \exp\left(\eta\rho A_d^{\varepsilon-\frac{1}{2}}n^{\frac{1}{2}-
    \varepsilon}\right).
\end{equation}
The equation
\begin{equation*}
     \frac{\sqrt{2}}{K_8\gamma A_d^{\frac{1}{4}}\sqrt{d}}(1+F_2)\exp\left(\frac{3-d}{2}\log\alpha +F_1\right)n= \exp\left(\eta\rho A_d^{\varepsilon-\frac{1}{2}}n^{\frac{1}{2}-
    \varepsilon}\right)
\end{equation*}
has two positive solutions $\sigma_0 < \sigma_1$ for $n$ and \eqref{N8} is satisfied for all $n\geq \sigma_1$. We thus define $N_8 := \lceil \sigma_1 \rceil$, which we calculate using a root finding program in Sagemath for each $4\leq d \leq 61$.  Thus we have $S_8 \leq K_8m_d(n)$ for all $n\geq N_8$.
 
Then setting
\begin{equation}\label{N_qdef}N_q=\max \{N_4,N_5,N_6,N_7,N_8\}\end{equation} 
gives the desired result.
\end{proof}

\section{Discussion of computations}\label{computations}

In this section we present the computation of the positive integers $N(d)$ for which $n\geq N(d)$ guarantees $q_d^{(2)}(n)\geq Q_d^{(2,-)}(n)$, the recursive algorithms used to compute exact values of $Q_d^{(2)}(n)$ and $q_d^{(2)}(n)$ for $1\leq n\leq N(d)$, and how we compute the difference $q_d^{(2)}(n)-Q_d^{(2,-)}(n)$ for necessary values of $n$ to justify that $Q_d^{(2,-)}(n)\leq q_d^{(2)}(n)$ for every positive integer $n$.  The code we use is a modified version of the C++ code used by Alfes et al. \cite{AJLO} to compute values of $Q_d^{(1)}(n)$ and $q_d^{(1)}(n)$. 

\subsection{Determining bounds $N(d)$}

We first discuss our computation of $N(d)$ when $4\leq d \leq 61$ for which $n\geq N(d)$ guarantees that $q_d^{(2)}(n)\geq Q_d^{(2,-)}(n)$.  This requires a choice of values for several parameters subject to certain conditions, as well as a choice of values for the weights $K_i$, as described in Section \ref{bounds}.

The choices which have conditions that do not depend on the parity of $d$ are $0<\varepsilon<\frac{1}{2}$, $\varepsilon_2>\frac{1}{3}$, and $\varepsilon_2>\varepsilon$ which arise in the proof of Lemma \ref{qdbounds}, and also the computation of $|f_{err}^{max}|$ which depends on a choice of $\frac{3}{8}<c<\frac{1}{2}$ (see \cite{DKT_proc}).

For all $4\leq d\leq 61$ we choose the values given in Table \ref{tab:vals} based on experimentation in Sagemath, and in the case of $\xi$ redefine it (it is previously defined in \eqref{xi}), since the choice below is simpler and overestimates the error term $R_d(n)$.

\begin{table}[ht]
\[
\begin{tabular}{|c|c|c|c|}
\hline
$c$ & $\varepsilon$ & $\varepsilon_2$ & $\xi$ \\  
\hline
0.37501 & 0.11 & 1 & 0.224 \\
\hline
\end{tabular}
\]
\caption{Values of $c$, $\varepsilon$, $\varepsilon_2$, and $\xi$ for all $4\leq d\leq 61$.}
\label{tab:vals}
\end{table}

From the proof of Lemma \ref{Qdbounds} we require $0<\delta<\frac{1}{2}$.  Our choice of $\delta$ depends on the parity of $d$ and is given in Table \ref{tab:delta}.  

\begin{table}[ht]
\[
\begin{tabular}{|c|c|}
\hline
$d$ & $\delta$ \\
\hline \hline
$d$ even  & $1/3$ \\  
\hline
$d$ odd  & $1/80$ \\  
\hline
\end{tabular}
\]
\caption{Values of $\delta$ for $4\leq d\leq 61$ based on parity of $d$.}
\label{tab:delta}
\end{table}

For even $4\leq d\leq 60$, based on Sagemath experimentation and the relative sizes of $N_i$ for $1\leq i\leq 8 $ when all weights $K_i$ set to $1$, we choose values for $K_i$ based on the parity of $d$ as given in Table \ref{tab:K}. 

\begin{table}[ht]
\[
\begin{tabular}{|c|c|c|c|c|c|c|c|c|}
\hline
$d$ & $K_1$ & $K_2$ & $K_3$ & $K_4$ & $K_5$ & $K_6$ & $K_7$ & $K_8$ \\
\hline \hline
$d$ even  & $1/800$ & $1/800$ & $1/2$ & $1/800$ & $1/800$ & $1/800$ & $1/800$ & $394/800 $ \\ 
\hline
$d$ odd  & $1/800$ & $1/8$ & $1/8$ & $1/800$ & $1/800$ & $1/800$ & $1/800$ & $595/800$ \\ 
\hline
\end{tabular}
\]
\caption{Values of weights $K_i$ for $4\leq d\leq 61$ based on parity of $d$.}
\label{tab:K}
\end{table}

Using the values given in Tables \ref{tab:vals}, \ref{tab:delta}, and \ref{tab:K}, we compute $N(d)$ for each $4\leq d\leq 61$ as given in Table \ref{tab:N(d)}.  Notably, for even $6\leq d \leq 60$, and odd $9\leq d\leq 61$, we have
\[
N(d)<10^7,
\]
and the only larger values are

\begin{align*}
N(4) & < 3.9\times 10^7, \\
N(5) & < 1.5 \times 10^8, \\
N(7) & < 1.7\times 10^7. 
\end{align*}

\begin{table}[ht]
\[    
    \begin{tabular}{|c|c||c|c||c|c|}
    \hline
       $d$  & $N(d)$  & $d$ & $N(d)$ & $d$ & $N(d)$ \\
       \hline \hline
        $4$ & $38133800$ & $24$ & $1168195$ & $44$ & $3300632$ \\
         \hline
         5 &142685922    &  25&1174519 & 45&3257697 \\
         \hline
          6& 2270342 &  26& 1331711& 46&3576985\\
         \hline
          7& 16962519&27  & 1334627   & 47&  3527299  \\
         \hline
          8& 577857 &  28& 1505944 & 48&3865326\\
         \hline
          9&  4661719   &  29&   1505109  &49& 3808560  \\
         \hline
          10& 314268 & 30& 1691018&50&5165784\\
         \hline
          11&  1886829   & 31& 1686090   &51&  4101610  \\
         \hline
          12&405797  & 32& 1887055 &52&4478487\\
         \hline
          13&  949272   & 33&  1877697   &53& 4406575  \\
         \hline
          14& 507346 & 34& 2094182 &54&4803561\\
         \hline
          15&  547612   &  35&  2080058   &55& 4723585  \\
         \hline
          16& 618979  & 36& 2312526 &56&5141132\\
         \hline
          17&  635395   & 37&  2293302  &57& 5052765  \\
         \hline
          18&740779  & 38& 2542214 &58&5491330\\
         \hline
         19 &  755215   & 39&  2517558   &59&  5394245  \\
         \hline
          20&  872843& 40& 2783376&60&5854276\\
          \hline
           21&  884932   &  41&  2752957  &61 &  5748150  \\
         \hline
          22& 1015278 &  42&3036139 & &\\
          \hline
          23&  1024661    &  43&  2999626   & &\\
         \hline
    \end{tabular}
 \]   
    \caption{Values of $N(d)$ for each $4\leq d\leq 61$.}
    \label{tab:N(d)}
\end{table}

\subsection{Algorithms to compute $Q_d^{(2)}(n)$, $Q_d^{(2,-)}(n)$, and $q_d^{(2)}(n)$}

The algorithm we use to compute values of $Q_d^{(2)}(n)$ for large $n$ is recursive.  The recursive step to generate $Q_d^{(2)}(n)$ relies on generating each of the allowable parts up to the maximum allowable sized part. We will refer to the allowable parts as a sequence, $(a_k)_{k=0}^{\infty}$ and order them from smallest to largest. So, $a_0=2$, $a_1=d+1$, $a_2=d+5,$ $a_3=2d+4$, and so on. 

Denote by $Q_k(n)$ the number of partitions of $n$  with parts from $(a_k)_{k=0}^{\infty}$  having largest part $a_k$. Then, split these into two sets based on whether $a_k$ appears exactly once in the partition or more than once. Let $A$ be the set of these partitions for which there is exactly one occurrence of $a_k$, and $B$ those for which there are two or more occurrences of $a_k$. Consider a partition of $n$ in $A$; if we remove the part $a_k$, we obtain a partition of $n-a_k$ whose largest part is at most $a_{k-1}$. So, the number of partitions in $A$ is equal to $\sum_{i=0}^{k-1}Q_i(n-a_k)$.  Next, consider a partition of $n$ in $B$; if we remove the part $a_k$, then since there were at least two occurrences of $a_k$ in the partition, we obtain a partition of $n-a_k$ still with largest part $a_k$. Thus, the number of partitions in $B$ is equal to $Q_k(n-a_k)$. Since $A$ and $B$ are disjoint and their union is the set of all partitions of $n$ with largest part $a_k$, we have the recursion
\begin{equation}
    Q_k(n)=Q_k(n-a_k)+\sum_{i=0}^{k-1}Q_i(n-a_k).
\end{equation}
We continue this process, running through all allowable parts.  We can also generate $Q_d^{(2,-)}(n)$ in this way by skipping $a_1=d+1$ in the recursion.

The algorithm to generate $q_d^{(2)}(n)$ is also recursive. To find exact values for $q_d^{(2)}(n)$, we use the fact that there is a bijection between partitions of $n$ into $k$ parts and $d$-distinct partitions of $n+d\binom{k}{2}+2k$ into $k$ parts which are all greater than $1$. Leveraging this bijection, we instead compute the total number of partitions of $n$ into exactly $k$ parts, which we will denote by $p_k(n)$. By summing the values of $p_k(n)$ appropriately, we can find $q_d^{(2)}(n)$ for $1\leq n\leq N(d)$.  To compute $p_k(n)$ we use the recursion
\begin{equation*}
    p_k(n)=p_{k-1}(n-1)+p_k(n-k),
\end{equation*}
which is explained by splitting the partitions counted by $p_k(n)$ into those which have $1$ as a part, and those which don't.  Those which have $1$ as a part can be enumerated by $p_{k-1}(n-1)$ via the bijection of removing a part of size $1$ to obtain a partition of $n-1$ into $k-1$ parts.  Those which do not contain $1$ as a part are enumerated by $p_k(n-k)$ via the bijection of removing $1$ from each of the $k$ parts to obtain a partition of $n-k$ into $k$ parts.

\subsection{Computing}
Using the described recursive algorithms and the High Performance Computing Cluster (HPC) at Oregon State University we computed exact values of $\Delta_d^{(2)}(n)$ and $\Delta_d^{(2,-)}(n)$ for $1\leq n\leq N(d)$ and $d$ as in Theorem \ref{fullKP}. For all $4\leq d \leq 61$ we had success computing $\Delta_d^{(2)}(n)$ and $\Delta_d^{(2,-)}(n)$ up to $n=1.9 \times 10^7$. For $N(d)\leq 10^7$, it takes a few hours to compute $\Delta_d^{(2)}(n)$ for all $n\leq N(d).$ The time increases dramatically as $N(d)$ increases. However, it still takes less than two weeks to compute $\Delta_d^{(2)}(n)$ up to $n=1.9 \times 10^7$ on the HPC.

Unfortunately, it takes too long to compute $\Delta_4^{(2)}(n)$ and $\Delta_5^{(2)}(n)$ up to our largest bounds $N(4)$ and $N(5)$ which is why these cases are excluded from Theorem \ref{fullKP}.

Upon computing $q_d^{(2)}(n)-Q_d^{(2)}(n)$ for all even $6\leq d\leq 60$ and $1\leq n\leq N(d),$ we have shown that $\Delta_d^{(2)}(n)\geq 0$.  When $d$ is odd however, there are values of $n$ for which $q_d^{(2)}(n)-Q_d^{(2)}(n)$ is negative. However these values occur precisely when $n=d+1$, $d+3$, and $d+5$, which proves Theorem \ref{CKK_extension}. 

Our computations confirm that $\Delta_d^{(2,-)}(n)\geq 0$ for all $n\geq 1$ when  $6\leq d\leq 61$.  Thus, we have proven Theorem \ref{fullKP}. Additionally, we have computed that $\Delta_d^{(2,-)}(n)\geq 0$ when $d\in\{3,4,5\}$ for $1\leq n\leq 10^7$.

\section{Concluding Remarks}\label{conclusion}

This paper, along with the work of Duncan et al. \cite{DKST} settles Kang and Park's conjecture for all values of $d$ except $3$, $4$, and $5$. We note that for $d=4$ it remains only to show that $\Delta_d^{(2)}(n)\geq 0$ for $10^7 < n < 3.8 \times 10^7$, and for $d=5$ it remains to show that $\Delta_d^{(2,-)}(n)\geq 0$ for $10^7 < n < 1.5 \times 10^8$.  However for $d=3$ asymptotic bounds as in Alfes et al.\cite{AJLO} have not yet been worked out and we suspect that an extension of their results for $Q_3^{(1)}(n)$ to overestimate $Q_3^{(2,-)}(n)$ may produce a $N(3)$ that is too large to compute $\Delta _3^{(2)}(n)$ for all $1\leq n\leq N(3)$. It would be interesting to see whether a combinatorial approach could prove the $d=3 $case as thus far that approach has not yet succeeded. 

Computational constraints also arise when we attempt to further extend Cho, Kang, and Kim's result \cite[Theorem 1.1]{CKK} to $62\leq d\leq 252$, $d\neq 126$. In general the estimations for the constants involved in the error terms of the asymptotics for $Q_d^{(2)}(n)$ and $q_d^{(2)}(n)$ leave room for improvement, and doing so would allow for less computational constraint and align more closely to what we observe in computations.  Of course, other methods may prove more fruitful.

We further note that Duncan et al. \cite{DKST} also investigated generalizing Kang and Park's conjecture \eqref{KP_conj} to general $a$.  Recent progress on these generalizations has been done by Inagaki and Tamura \cite{IT} as well as Armstrong, Ducasse, Meyer, and the second author \cite{ADMS}.  In particular, it is now known that $\Delta_d^{(3,-)}(n)$ holds for all but finitely many cases.  The methods described in this paper could perhaps be generalized to prove these as well.


\begin{thebibliography}{10}

\bibitem{Alder_nonex}
Henry~L. Alder.
\newblock The nonexistence of certain identities in the theory of partitions
  and compositions.
\newblock {\em Bulletin of the American Mathematical Society}, 54(8):712 --
  722, 1948.

\bibitem{Alder_conj}
Henry~L. Alder.
\newblock Research problems, no. 4.
\newblock {\em Bull. Amer. Math. Soc.}, 62(1):76, 1956.

\bibitem{AJLO}
Claudia Alfes, Marie Jameson, and Robert~J. Lemke~Oliver.
\newblock Proof of the {A}lder-{A}ndrews conjecture.
\newblock {\em Proc. Amer. Math. Soc.}, 139(1):63--78, 2011.

\bibitem{Andrews_Alder}
George~E. Andrews.
\newblock On a partition problem of {H}. {L}. {A}lder.
\newblock {\em Pacific J. Math.}, 36:279--284, 1971.

\bibitem{andrews_book}
George~E. Andrews.
\newblock {\em The theory of partitions}.
\newblock Cambridge Mathematical Library. Cambridge University Press,
  Cambridge, 1998.
\newblock Reprint of the 1976 original.

\bibitem{Apostol}
Tom~M. Apostol.
\newblock {\em Introduction to analytic number theory}.
\newblock Undergraduate Texts in Mathematics. Springer-Verlag, New
  York-Heidelberg, 1976.

\bibitem{ADMS}
Liam Armstrong, Bryan Ducasse, Thomas Meyer, and Holly Swisher.
\newblock Generalized {A}lder-type partition inequalities.
\newblock {\em Electron. J. Combin.}, 30(2):Paper No. 2.36, 16, 2023.

\bibitem{BFOR}
Kathrin Bringmann, Amanda Folsom, Ken Ono, and Larry Rolen.
\newblock {\em Harmonic {M}aass forms and mock modular forms: theory and
  applications}, volume~64 of {\em American Mathematical Society Colloquium
  Publications}.
\newblock American Mathematical Society, Providence, RI, 2017.

\bibitem{CKK}
Haein Cho, Soon-Yi Kang, and Byungchan Kim.
\newblock Alder-type partition inequality of general level.
\newblock {\em arXiv preprint arXiv:2307.14048v3}, 2023.

\bibitem{DKST}
Adriana~L. Duncan, Simran Khunger, Holly Swisher, and Ryan Tamura.
\newblock Generalizations of {A}lder's conjecture via a conjecture of {K}ang
  and {P}ark.
\newblock {\em Res. Number Theory}, 7(1):Paper No. 11, 26, 2021.

\bibitem{DKT_proc}
Adriana~L. Duncan, Simran Khunger, and Ryan Tamura.
\newblock Generalizations of {A}lder's conjecture via a conjecture of {K}ang
  and {P}ark.
\newblock In {\em Proceedings of the 2020 REU in Mathematics and Theoretical
  Computer Science at Oregon State University}, pages 1--42, 2020.
\newblock
  \url{https://sites.science.oregonstate.edu/math_reu/proceedings/2020.html}.

\bibitem{IT}
Ryota Inagaki and Ryan Tamura.
\newblock On generalizations of a conjecture of {K}ang and {P}ark.
\newblock {\em Res. Number Theory}, 9(3):Paper No. 51, 21, 2023.

\bibitem{KangPark}
Soon-Yi Kang and Eun~Young Park.
\newblock An analogue of {A}lder-{A}ndrews conjecture generalizing the 2nd
  {R}ogers-{R}amanujan identity.
\newblock {\em Discrete Math.}, 343(7):111882, 6, 2020.

\bibitem{Lehmer}
D.~H. Lehmer.
\newblock Two nonexistence theorems on partitions.
\newblock {\em Bull. Amer. Math. Soc.}, 52:538--544, 1946.

\bibitem{Meinardus_asymptotics}
G{\"u}nter Meinardus.
\newblock Asymptotische aussagen {\"u}ber partitionen.
\newblock {\em Mathematische Zeitschrift}, 59(1):388--398, 1954.

\bibitem{Meinardus_partitions}
G{\"u}nter Meinardus.
\newblock {\"U}ber partitionen mit differenzenbedingungen.
\newblock {\em Mathematische Zeitschrift}, 61(1):289--302, 1954.

\bibitem{SCHUR}
Issai Schur.
\newblock {\em Gesammelte {A}bhandlungen. {B}and {III}}.
\newblock Springer-Verlag, Berlin-New York, 1973.
\newblock reprint containing original reference published in 1926.

\bibitem{whittaker1927course}
Edmund~Taylor Whittaker and George~Neville Watson.
\newblock {\em A course of modern analysis: an introduction to the general
  theory of infinite processes and of analytic functions; with an acount of the
  principal transcendental functions}.
\newblock University Press, 1927.

\bibitem{Yee_04}
Ae~Ja Yee.
\newblock Partitions with difference conditions and {A}lder's conjecture.
\newblock {\em Proc. Natl. Acad. Sci. USA}, 101(47):16417--16418, 2004.

\bibitem{Yee_08}
Ae~Ja Yee.
\newblock Alder's conjecture.
\newblock {\em J. Reine Angew. Math.}, 616:67--88, 2008.

\end{thebibliography}
\end{document}